\documentclass[10pt, a4paper,reqno]{amsart}

\usepackage{amsmath}
\usepackage{amsfonts}
\usepackage{amssymb}
\usepackage{amsthm}
\usepackage{latexsym}
\usepackage[active]{srcltx}
\usepackage{color}

\numberwithin{equation}{section}

\newtheorem{theorem}[equation]{Theorem}
\newtheorem{lemma}[equation]{Lemma}
\newtheorem{proposition}[equation]{Proposition}

\theoremstyle{definition}
\newtheorem{definition}[equation]{Definition}

\theoremstyle{remark}
\newtheorem{remark}[equation]{Remark}

\def\kint_#1{\mathchoice%
          {\mathop{\kern 0.2em\vrule width 0.6em height 0.69678ex depth -0.58065ex
                  \kern -0.8em \intop}\nolimits_{\kern -0.4em#1}}%
          {\mathop{\kern 0.1em\vrule width 0.5em height 0.69678ex depth -0.60387ex
                  \kern -0.6em \intop}\nolimits_{#1}}%
          {\mathop{\kern 0.1em\vrule width 0.5em height 0.69678ex depth -0.60387ex
                  \kern -0.6em \intop}\nolimits_{#1}}%
          {\mathop{\kern 0.1em\vrule width 0.5em height 0.69678ex depth -0.60387ex
                  \kern -0.6em \intop}\nolimits_{#1}}}
\def\vintslides_#1{\mathchoice%
          {\mathop{\kern 0.1em\vrule width 0.5em height 0.697ex depth -0.581ex
                  \kern -0.6em \intop}\nolimits_{\kern -0.4em#1}}%
          {\mathop{\kern 0.1em\vrule width 0.3em height 0.697ex depth -0.604ex
                  \kern -0.4em \intop}\nolimits_{#1}}%
          {\mathop{\kern 0.1em\vrule width 0.3em height 0.697ex depth -0.604ex
                  \kern -0.4em \intop}\nolimits_{#1}}%
          {\mathop{\kern 0.1em\vrule width 0.3em height 0.697ex depth -0.604ex
                  \kern -0.4em \intop}\nolimits_{#1}}}

\newcommand{\R}{\mathbb{R}}
\newcommand{\Rn}{\mathbb{R}^d}
\newcommand{\Z}{\mathbb{Z}}

\newcommand{\abs}[1]{\lvert#1\rvert}

\renewcommand{\limsup}{\operatornamewithlimits{lim \, sup}}

\renewcommand{\l}{\left}
\renewcommand{\r}{\right}

\def\Xint#1{\mathchoice
{\XXint\displaystyle\textstyle{#1}}%
{\XXint\textstyle\scriptstyle{#1}}%
{\XXint\scriptstyle\scriptscriptstyle{#1}}%
{\XXint\scriptscriptstyle\scriptscriptstyle{#1}}%
\!\int}

\def\XXint#1#2#3{{\setbox0=\hbox{$#1{#2#3}{\int}$}
\vcenter{\hbox{$#2#3$}}\kern-.5\wd0}}

\def\dashint{\Xint-}

\newenvironment{list2}{
  \begin{list}{$\bullet$}{%
      \setlength{\itemsep}{0in}
      \setlength{\parsep}{0in} \setlength{\parskip}{0in}
      \setlength{\topsep}{0in} \setlength{\partopsep}{0in} 
      \setlength{\leftmargin}{0.2in}}}{\end{list}}

\title[Representation and asymptotic behavior of solutions]{Representation of solutions and large-time behavior for fully nonlocal diffusion equations}
\author{Jukka Kemppainen, Juhana Siljander and Rico Zacher}

\begin{document}



\begin{abstract}
We study the Cauchy problem for a nonlocal heat equation, which is of fractional order both in space 
and time. We prove four main theorems: 
\begin{itemize}
\item[(i)] a representation formula for classical solutions,
\item[(ii)] a quantitative decay rate at which the solution tends to the fundamental solution,
\item[(iii)] optimal $L^2$-decay of mild solutions in all dimensions,
\item[(iv)] $L^2$-decay of weak solutions via energy methods.
\end{itemize}

The first result relies on a delicate analysis of the definition of classical solutions. After proving the representation formula we carefully analyze the integral representation to obtain the quantitative decay rates of (ii). 

Next we use Fourier analysis techniques to obtain the optimal decay rate for mild solutions. Here we encounter the {\em critical dimension phenomenon} where the decay rate attains the decay rate of that in a bounded domain for large enough dimensions. Consequently, the decay rate does not anymore improve when the dimension increases. The theory is markedly different from that of the standard caloric functions and this substantially complicates the analysis. 

Finally, we use energy estimates and a comparison principle to prove a quantitative decay rate for weak solutions defined via a variational formulation. Our main idea is to show that the $L^2$--norm is actually a subsolution to a purely time-fractional problem which allows us to use the known theory to obtain the result.



\end{abstract}

\maketitle

\section{Introduction}

We study the Cauchy problem for the diffusion equation
\begin{equation}\label{prob}
\partial_t^{\alpha} (u(t,x)-u_0)+\mathcal L u(t,x)=f(t,x)\quad 
\mathrm{in}\quad\R_+ \times \Rn,\ \  0 < \alpha \le 1,
\end{equation} 
where \(u_0(x)=u(0,x)\) is the initial condition, $\partial_t^\alpha$ denotes the Riemann-Liouville fractional derivative if $\alpha\in (0,1)$ and $\mathcal L$ is a nonlocal elliptic operator of order $\beta \in (0,2]$. A standard example is the fractional Laplacian $\mathcal L = (-\Delta)^{\frac{\beta}{2}}$. The equation is nonlocal both in space and time and we call such a parabolic equation {\em a fully nonlocal diffusion equation}.

Our emphasis is on the decay properties, and for the space-fractional heat diffusion such questions have been studied, for instance, by Chasseigne, Chaves and Rossi in~\cite{ChasChavRoss06} as well as by Ignat and Rossi in~\cite{IgnaRoss09}. For a more comprehensive account of the asymptotic theory in case $\alpha=1$, we refer to~\cite{Ross09}.  The decay of solutions and behavior of the Barenblatt solution for the space-fractional porous medium equation has, in turn, been studied by Vazquez in~\cite{Vazq14}. In the present paper, we extend these developments -- concerning the fundamental solutions, representation formulas and decay properties -- to the above fully nonlocal equation. For the case $\beta=2$, see earlier works by Vergara and Zacher in~\cite{VergZach15} and by Vergara and the present authors in~\cite{KempSiljVergZach14}. For the regularity theory of nonlocal equations in case $\alpha=1$ or $\beta=2$, we refer to~\cite{CaffSilv12, CaffChanVass11, FelsKass13, BonfVazq13, BonfVazq14, KimLee13,BarlBassChenKass09, Zach13a, Zach13b} and the references therein.

Nonlocal PDE models arise directly, and naturally, from applications. 
Time fractional diffusion equations are closely related to a class of Montroll-Weiss
continuous time random walk (CTRW) models and have become
one of the standard physics approaches to model {\em anomalous diffusion} processes~\cite{DragKlaf00, CompCace98,Hilf03,MeerBensScheBeck02}. For a detailed derivation of these equations from physics principles and for further applications of such models we refer to the expository review article of Metzler and Klafter in~\cite{MetzKlaf00}. The fractional Laplacian arises in the modelling of jump processes and also in quantitative finance as a model for pricing American options~\cite{ContTank04, Silv05}. The fully nonlocal diffusion equation, in particular, has been used in diffusion models, for instance, in~\cite{CartCast07} and~\cite{CompCace98}. 

Despite their importance for applications, the mathematical study of fully nonlocal diffusion problems of type~\eqref{prob} is relatively young. In a very recent paper Allen, Caffarelli and Vasseur~\cite{AlleCaffVass15} have studied the regularity of weak solutions to such problems. Even more recently, simultaneously to our work, Kim and Lim~\cite{KimLim15} have considered the behavior of fundamental solutions, whereas Cheng, Li and Yamamoto~\cite{ChenLiYama15} have studied other aspects of the asymptotic theory. Apart from these papers, the study of the parabolic problem has mostly concentrated on the aforementioned cases $\alpha=1$ or $\beta=2$. 

We point out that the nonlocal in time term in (\ref{prob}), with $\partial_t^\alpha$ being the Riemann-Liouville
fractional derivation operator, coincides (for sufficiently smooth $u$) with the Caputo fractional derivative of $u$,
see (\ref{caputo}) below. The formulation with Riemann-Liouville fractional derivative has the
advantage that a priori less regularity is required on $u$ to define the nonlocal operator. In particular, our formulation is exactly the one which naturally arises from physics applications, see for instance~\cite[equation (40)]{MetzKlaf00}. 


Our first main result considers a representation formula for classical solutions of the Cauchy problem for equation~\eqref{prob} with $\mathcal L = (-\Delta)^{\frac\beta2}$. 
In the process, we calculate the exact behavior of the fundamental solutions.

Next, we show that the mild solutions, which are defined through the representation formula whenever its integrals are finite, tend to the fundamental solutions $Z$ and $Y$ -- corresponding to the initial and forcing data, respectively -- in $L^p$ with {\em quantitative} decay rates. Such results are nontrivial already for standard caloric functions, especially in the case of a non-vanishing forcing term. In particular, the proof requires a delicate analysis of the problem as well as gradient estimates for the fundamental solutions which can only be represented via so called Fox $H$-fucntions. In the analysis of these special functions we use number theoretic tools to obtain their behavior up to the first derivatives. A particular difficulty in all the analysis is caused by the fact that the fundamental solutions $Z$ and $Y$ have singularities also for positive times. This causes integrability problems and requires a delicate analysis.

We continue to study decay results by two additional approaches. In the first one, we use Fourier techniques to build {\em optimal} $L^2$--decay estimates for mild solutions of the aforementioned Cauchy problem. Contrary to the standard caloric functions, the decay rate does not improve with high enough dimensions, but there exists a critical dimension at which the decay rate of bounded domains is achieved. This {\em critical dimension phenomena} is brought by the introduction of the fractional Riemann-Liouville time-derivative and such behavior is not observed in the case $\alpha=1$. This also substantially complicates the analysis and we are required to use Riesz potential estimates to obtain the decay results. Thus the theory is markedly different from that of the standard heat equation.

Finally, we turn into studying the decay of weak solutions where we can consider operators $\mathcal L$ with general measurable kernels. We show that the $L^2$-norm of a weak solution, which is defined in a variational formulation, is a subsolution to a purely time-fractional equation. On the other hand, the exact behavior of the solutions for such problems is well-known and, therefore, we may use the comparison principle to conclude the result -- even in such a general context.  While our method gives the optimal decay rate in the case $\alpha=1$, the energy methods used in the proof cannot discriminate between large and small dimensions. Consequently, we are not able to obtain the non-smooth decay behavior -- and the consequent critical dimension phenomenon -- with respect to the dimension. Thus, it remains an open question whether our decay result is optimal in this context.

\section{Preliminaries and main results} \label{section:preliminaries}

\subsection{Notations and definitions}

Let us first fix some notations. We denote the space of \(k\)-times 
continuously differentiable functions by \(C^k\) and \(C^0:=C\). 

The Riemann-Liouville fractional integral of order \(\alpha \ge 0\) is defined for $\alpha = 0$ as \(J^0:=I\), where \(I\) denotes the identity operator, and for $\alpha >0$ as
\begin{equation}\label{Riemann-Liouville_int}
J^\alpha 
f(t)=\frac{1}{\Gamma(\alpha)}\int_0^t(t-\tau)^{\alpha-1}f(\tau)d\tau = (g_\alpha * f)(t),
\end{equation}
where 
\[
g_\alpha(t)=\frac{t^{\alpha-1}}{\Gamma(\alpha)}
\]
is the Riemann-Liouville kernel and $*$ denotes the convolution in time. We denote the convolution in space by $\star$ and the double convolution in space and time by $\hat*$. 

The Riemann-Liouville fractional derivative of order $0<\alpha  < 1$ is 
defined by
\begin{equation}\label{Riemann-Liouville}
\partial_t^{\alpha}f(t)=\frac{d}{dt} J^{1-\alpha}f(t)=\frac{d}{dt}\frac{1}{\Gamma(1-\alpha)}
\int_0^ { t } (t-\tau)^ { -\alpha } f(\tau)d\tau.
\end{equation}
Observe that for sufficiently smooth $f$ and $\alpha\in (0,1)$
\begin{equation} \label{caputo}
\partial_t^{\alpha}(f-f(0))(t)=(J^{1-\alpha} f')(t)=:{}^cD_t^\alpha f(t),
\end{equation}
the so-called Caputo fractional derivative of $f$.
In case $\alpha=1$, we have the standard time derivative. 

Let 
\[
\widehat{u}(\xi)=\mathcal{F}(u)(\xi)=(2\pi)^{-d/2}\int_{\Rn}\mathrm{e}^{
-ix\cdot\xi}f(x)dx
\] 
and
\[
\mathcal{F}^{-1}(u)(\xi):=\mathcal{F}(u)(-\xi)
\]
denote the Fourier and inverse Fourier transforms of $u$, respectively. We define the fractional Laplacian 
as 
\[
(-\Delta)^{\beta/2}u(x)=\mathcal{F}^{-1}_{\xi\to x}(|\xi|^\beta
\widehat{u}(\xi)).
\]

Next we define the concept of a classical solution 
of~\eqref{prob}, with $\mathcal L = (-\Delta)^{\beta/2}$, given with an initial condition \(u(0,x)=u_0(x)\). 

\begin{definition}
Let $0 < \alpha \le 1$ and $0 < \beta \le 2$. Suppose $u_0 \in C(\Rn)$ and $f \in C([0, \infty)\times\Rn)$.
Then a function $u \in C([0, \infty) \times \Rn) $ is a {\em classical solution} of the Cauchy problem
\begin{equation}\label{classical_equation}
\begin{split}
\begin{cases}
\ \partial_t^\alpha (u(t,x)-u_0)+(-\Delta)^{\beta/2}u(t,x) &= 
\ \,f(t,x),\quad\mathrm{in}\ \ (0,\infty)\times\Rn, \\
\qquad \qquad \qquad \qquad \qquad \qquad u(0,x) &= 
\ \, u_0(x),\quad\,\,\mathrm{in}\ \ \Rn,
\end{cases}
\end{split}
\end{equation}
if
\begin{itemize}

\item[(i)] $\mathcal{F}^{-1}_{\xi\to x}(|\xi|^\beta
\widehat{u}(\xi))$ defines a continuous function of $x$ for each $t>0$,

\item[(ii)] for every $x \in \Rn$, the fractional integral $J^{1-\alpha}u$, as 
defined in~\eqref{Riemann-Liouville_int}, is continuously differentiable with 
respect to $t>0$, and

\item[(iii)] the function $u(t,x)$ satisfies the 
integro-partial differential equation of~\eqref{classical_equation} for every 
$(t,x) \in (0, \infty) \times \Rn$ and the initial condition 
of~\eqref{classical_equation} for every \(x\in\Rn\). 
\end{itemize}
\end{definition}

We remark that under appropriate regularity conditions on the data, existence and uniqueness of strong $L^p$-solutions of (\ref{classical_equation}) follows from the results in~\cite{Zach05}, which are formulated in the framework
of abstract parabolic Volterra
equations, see also the monograph~\cite{Prus93}.

Next we turn in to the weak solutions to equation~\eqref{prob}. In place of the fractional Laplacian we will consider a more general class of elliptic operators. In this context, we avoid using the Fourier transform and the corresponding definition for the  fractional Laplacian is given by its singular integral representation
\begin{equation}\label{variational_laplacian}
(-\Delta)^{\frac{\beta}{2}}u(x)=c(d,\beta)\mathrm{P.V.}\int_{\Rn}\frac{
u(x)-u(y) } { |x-y|^ { d+\beta}}dy
\end{equation}
where P.V. stands for the Cauchy principal value and \(c\) is a constant. 
In~\cite{Silv07} it is shown that \((-\Delta)^{\frac{\beta}{2}}u(x)\) is 
a continuous function  whenever \(u\) is locally in \(C^2(\Rn)\) 
and 
\[
\int_{\Rn}\frac{|u(x)|}{1+|x|^{d+\beta}}dx<\infty.
\]
We will study the weak formulation where we define the operator through a bilinear form. We begin by setting up the problem. 

We define the fractional Sobolev space $W^{\frac\beta2, 2}(\Rn)$ for $\beta \in (0,2)$ as
\[
W^{\frac\beta2, 2}(\Rn):= \l\{v \in L^2(\Rn): \frac{|v(x)-v(y)|}{|x-y|^{\frac{d+\beta}{2}}} \in L^2(\Rn\times \Rn)\r\}
\]
endowed with the norm
\[
\|v\|_{W^{\frac\beta2, 2}(\Rn)}:=\l(\int_{\Rn}|v|^2 \, dx + \int_{\Rn}\int_{\Rn} \frac{|v(x)-v(y)|^2}{|x-y|^{d+\beta}} \, dx \, dy\r)^{1/2}.
\]

Let $0 < \lambda \le \Lambda$ and define the {\em kernel} $K: \Rn  \times \Rn \to [0, \infty)$ to be a measurable function such that
\begin{equation}\label{K_assumption}
\frac\lambda{|x-y|^{d+\beta}} \le K(x,y) \le \frac\Lambda{|x-y|^{d+\beta}}
\end{equation}
for almost every $x, y \in \Rn$ and for some $\beta \in (0,2)$.  Consider the bilinear form
\[
\mathcal E(u,v) :=  \int_{\Rn}\int_{\Rn} K(x,y)[u(x)-u(y)] \cdot [v(x)-v(y)] \, dx \, dy
\]
for any $u, v \in W^{\frac\beta2, 2}(\Rn)$. Let $\varphi \in W^{\frac\beta2, 2}(\Rn)$. We now define an elliptic operator $\mathcal L$ by
\[
\langle \mathcal L  u, \varphi\rangle= \mathcal E(u, \varphi),
\]
where $\langle \cdot,\cdot \rangle$ stands for the duality pairing on $V'\times V$ with $V=W^{\frac\beta2, 2}(\Rn)$.
Observe that if $K(x,y)=c(d,\beta)|x-y|^{-d-\beta}$ the operator $\mathcal L$ 
defined here gives the fractional Laplacian of~\eqref{variational_laplacian}. 

We study the Cauchy problem for weak solutions of the equation
\begin{equation}\label{weak_equation}
\partial_t^\alpha(u-u_0)+\mathcal L u = 0. 
\end{equation}
In the case $\alpha=1$, a weak solution is defined in the classical way. 
Letting $T>0$, a natural parabolic function space for defining weak solutions on $[0,T]\times \R^d$ in the case $\alpha\in (0,1)$
is given by
\begin{align*}
F_\alpha(T) := \{ v \in L^{\frac{2}{1-\alpha}, \infty}([0, T]; L^2(\Rn)) \cap L^2([0, T]; W^{\frac\beta2, 2}(\Rn)) \ \ \text{such that}& \ \ \\
 g_{1-\alpha} * v \in C([0, T]; L^2(\Rn)) \ \ \text{and} \ \ (g_{1-\alpha}*v)|_{t=0}=0&\},
\end{align*}
cf.~\cite{Zach09}. 

The definition of weak solution (in the case $\alpha\in (0,1)$) is now the following. 
\begin{definition}\label{weak_solution_def}
Let $u_0 \in L^2(\Rn)$ and $u:[0, \infty) \times \R^d \to \R$ be such that for any $T>0$ we have 
$u|_{[0,T]\times \R^d} \in F_\alpha(T)$. Then we say that $u$ is a {\em 
weak solution} of equation~\eqref{weak_equation} with initial condition $u|_{t=0}=u_0$ if for all $T>0$
\begin{align*}
&\int_{0}^T\int_{\Rn} -[g_{1-\alpha}*(u(t,x)-u_0(x))]\partial_t\varphi(t,x) \, dx \, dt  \\
& + \int_0^T \int_{\Rn}\int_{\Rn} K(x,y)[u(t,x)-u(t,y)]\cdot[\varphi(t,x)-\varphi(t,y)] \, dx \, dy \, dt=0
\end{align*}
 for all test functions $\varphi \in W^{1,2}([0, T]; L^2(\Rn)) \cap L^2([0, T];W^{\frac\beta2, 2}(\Rn))$ with 
 $\varphi|_{t=T} = 0$ in $L^2(\R^d)$.
\end{definition} 

We recall that existence and uniqueness of weak solutions in $F_\alpha(T)$ has been studied in~\cite{Zach09},
even in a more general context. 

\subsection{Fox $H$-functions}

The Fox $H$-functions are special functions of a very general nature and there 
is a natural connection to the fractional calculus, since the
fundamental solutions of the Cauchy problem can be represented in terms of 
them. Since the asymptotic behavior of the Fox \(H\)-functions can be found 
from the literature, the Fox \(H\)-functions have a crucial role also in our 
asymptotic analysis. We collect here some basic facts on these functions.

Let us start with the definition. To simplify the notation we introduce
\[
(a_i,\alpha_i)_{k,p}:=((a_k,\alpha_k),(a_{k+1},\alpha_{k+1}),\dots,(a_p,
\alpha_p))
\]
for the set of parameters appearing in the definition of Fox $H$-functions. The 
Fox \(H\)-function is defined via a Mellin-Barnes type
integral as
\begin{equation}\label{mellin_barnes_integral}
H^{mn}_{pq} (z):= H^{mn}_{pq}\big[ z \big| \begin{smallmatrix}
                           (a_i,\alpha_i)_{1,p}\\
(b_j,\beta_j)_{1,q}                 
                                           \end{smallmatrix}
\big] =\frac{1}{2\pi i}\int_{\mathcal{L}}\mathcal{H}^{mn}_{pq}(s)z^{-s} 
ds,
\end{equation}
where 
\begin{equation}\label{mellin_transform}
\mathcal{H}^{mn}_{pq}(s)=\frac{\prod_{j=1}^m \Gamma(b_j+\beta_j s)\prod_{i=1}^n 
\Gamma(1-a_i-\alpha_i s)}{\prod_{i=n+1}^p\Gamma(a_i+\alpha_i 
s)\prod_{j=m+1}^q\Gamma(1-b_j-\beta_j s)}
\end{equation}
is the Mellin transform of the Fox H-function \(H_{pq}^{mn}\) and 
\(\mathcal{L}\) is the infinite contour in the complex plane which separates 
the poles
\begin{equation}\label{poles_b_jl}
b_{jl}=\frac{-b_j-l}{\beta_j}\quad (j=1,\dots,m;\, l=0,1,2,\dots)
\end{equation}
of the Gamma function \(\Gamma(b_j+\beta_j s)\) to the left of \(\mathcal{L}\) 
and the poles
\begin{equation}\label{poles_a_ik}
a_{ik}=\frac{1-a_i+k}{\alpha_i} \quad (i=1,\dots,n;\, k=0,1,2,\dots)
\end{equation}
to the right of \(\mathcal{L}\). 


We will need the following properties from Chapter 2 of~\cite{KilbSaig04}.

\begin{lemma}\label{Fox_properties}
Properties of Fox H-functions:
\begin{list2}
\vspace{.1in}
\item[(i)] \(
\frac{d}{d z}H^{mn}_{pq}\big[ z \big| \begin{smallmatrix}(a_i,\alpha_i)_{1,p}\\
(b_j,\beta_j)_{1,q} \end{smallmatrix} \big] = z^{-1}H^{m, n+1}_{p+1, q+1}\big[ z \big| \begin{smallmatrix} (0,1), &(a_i,\alpha_i)_{1,p}\\
(b_j,\beta_j)_{1,q}, &(1,1) \end{smallmatrix} \big]
\)
\vspace{.2in}

\item[(ii)] \(H_{pq}^{mn}\bigl[ z^{-1} \big| \begin{smallmatrix} (a_i,\alpha_i)_{1,p} \\
                                 (b_j,\beta_j)_{1,q}
                                \end{smallmatrix}
\big]=H_{qp}^{nm}\bigl[ z \big| \begin{smallmatrix} 
(1-b_j,\beta_j)_{1,q}\\
(1-a_i,\alpha_i)_{1,p}                                  
                                \end{smallmatrix}
\big] 
\)

\vspace{.2in}

\item[(iii)] \( H_{pq}^{mn}\big[ z \big| \begin{smallmatrix}
                            (a_i,\alpha_i)_{1,p-1},&(b_1,\beta_1)\\
(b_j,\beta_j)_{1,q} &
                           \end{smallmatrix}
\big]=H_{p-1,q-1}^{m-1,n}\big[ z \big| \begin{smallmatrix}
                            (a_i,\alpha_i)_{1,p-1}\\
(b_j,\beta_j)_{2,q}
                           \end{smallmatrix}
\big]. \)
\vspace{.2in}

\item[(iv)] \(
\partial_t^{\sigma} H_{pq}^{mn}\big[ t^{\rho} 
\big|\begin{smallmatrix}(a_i,\alpha_i)_{1,p}\\
(b_j,\beta_j)_{1,q} \end{smallmatrix} \big] =t^{-\sigma}H_{p+1, 
q+1}^{m,n+1}\big[ t^\rho \big| \begin{smallmatrix}
                            (0,\rho),&(a_i,\alpha_i)_{1,p}\\
(b_j,\beta_j)_{1,q}, &(\sigma, \rho)
                           \end{smallmatrix}
\big].
\)
\vspace{.2in}

\item[(v)] For $b>0$ and $x>0$ we have
\begin{align*}
&\int_0^\infty (xr)^{\omega}J_\eta(xr) H_{pq}^{mn}\big[ br^\tau \big|\begin{smallmatrix}(a_i,\alpha_i)_{1,p}\\
(b_j,\beta_j)_{1,q} \end{smallmatrix} \big] \, dr \\
&=\frac{2^\omega}{x}H_{p+2, q}^{m, n+1}\big[ b2^\tau x^{-\tau} \big|\begin{smallmatrix} \l(1-\frac{\omega+1}{2}-\frac{\eta}{2}, \frac{\tau}{2}\r), &(a_i,\alpha_i)_{1,p}, &\l(1 -\frac{\omega+1}{2}+\frac{\eta}{2}, \frac{\tau}{2}\r)\\
(b_j,\beta_j)_{1,q} \end{smallmatrix} \big]
\end{align*}

\vspace{.1in}
\item[(vi)] 
\(zH_{pq}^{mn}\big[ z \big| \begin{smallmatrix}
                            (a_i,\alpha_i)_{1,p}\\
(b_j,\beta_j)_{1,q} &
                           \end{smallmatrix} \big]=H_{pq}^{mn}\big[ z \big| \begin{smallmatrix}
                            (a_i+\alpha_i,\alpha_i)_{1,p}\\
(b_j+\beta_j,\beta_j)_{1,q} &
                           \end{smallmatrix}\big]. \)
\vspace{.1in}
\end{list2}

\begin{proof}
The first four properties are straightforward calculations based on the Mellin-Barnes integral representation of Fox $H$-functions. For properties (v) and (vi), we refer to Corollary 2.5.1 and Property 2.5 of~\cite{KilbSaig04}, respectively.
\end{proof}
\end{lemma}

\begin{remark}\label{rem_Htrans_formula}
There are some restrictive conditions on the parameters appearing in \((v)\) 
(for details, see~\cite[Corollary 2.5.1]{KilbSaig04}). The conditions are 
required for the convergence of the integrals. Since \((v)\) represents a 
Hankel 
transform formula for the Fox H-functions and the Fourier transform can be 
written as a Hankel transform, we will use \((v)\) in the proof of Theorem 
2.22 to calculate the inverse Fourier transform of \(\xi\mapsto 
|\xi|^\beta\widehat{Y}(t,\xi)\), which is not integrable in general. But since both 
sides of \((v)\) depend analytically on our choice of parameters, the identity 
\begin{equation}\label{distr_F_trans}
\langle \widehat{f},\varphi\rangle=\langle f,\widehat{\varphi}\rangle,
\end{equation} 
where \(\langle\cdot,\cdot\rangle\) denotes the duality pairing on 
\(\mathcal{S}'\times\mathcal{S}\) with $\mathcal{S}$ denoting the space of Schwartz functions on $\R^d$, allows us by analytic continuation to 
conclude that \((v)\) is valid even for a wider range of parameters than the 
range given by the restrictions in~\cite[Corollary 2.5.1]{KilbSaig04}. Then 
the Hankel transform formula \((v)\) has to be understood as the generalized 
Fourier transform~\eqref{distr_F_trans}. For details on 
generalizing integral identities we refer to~\cite{GelfShil64}.
\end{remark}

An important special case of the function \(H^{11}_{12}(-z)\) with the 
parameters \((a_i,\alpha_i)_{1,1}=(0,1)\) and 
\((b_j,\beta_j)_{1,2}=((0,1),(1-\alpha,\beta))\) is the two-parameter 
Mittag-Leffler function
\begin{equation}\label{mittag_asymp1}
E_{\alpha,\beta}(z)=\sum_{k=0}^\infty\frac{z^k}{\Gamma(\beta+\alpha k)}.
\end{equation}
It appears in the fundamental solutions of the Cauchy problem for 
integro-ordinary differential equations. Since the 
problem~\eqref{classical_equation} formally transforms into
\begin{equation}\label{classical_equation_four}
\begin{split}
\begin{cases}
\quad\ \partial_t^\alpha 
(\widehat{u}(t,\xi)-\widehat{u}_0(\xi))+|\xi|^{\beta}\widehat{u}(t,\xi) 
&=\quad \widehat{f}(t,\xi), \\
\qquad \qquad \qquad \qquad \qquad \qquad \widehat{u}(0,\xi) &= \quad 
\widehat{u}_0(\xi),
\end{cases}
\end{split}
\end{equation}
 the fundamental solutions in the Fourier domain can 
be formally expressed in terms of Mittag-Leffler functions.
It can be shown rigorously that in our case the fundamental solutions $Z$ and $Y$ 
satisfy 
\begin{equation}\label{fund_fourier_trans}
\widehat{Z}(\xi,t)=(2\pi)^{-d/2} E_{\alpha,1}(-\abs{\xi}^\beta t^{\alpha}).
\end{equation}
and
\begin{equation}\label{Y_fourier_trans}
\widehat{Y}(t, \xi)= (2\pi)^{-d/2} t^{\alpha-1}E_{\alpha, \alpha}(-|\xi|^\beta 
t^\alpha).
\end{equation}
The Mittag-Leffler function 
\(E_{\alpha,\alpha}(-x)\) is known to be completely monotone for \(x\in\R_+\) 
and it has the asymptotics 
\begin{equation}\label{ML_estimate}
E_{\alpha, \alpha}(-x) \sim \frac{1}{1+x^2}, \quad x \in \R_+ .
\end{equation}
For $E_{\alpha, 1}$ we have the asymptotic behavior
\begin{equation}\label{ML2_estimate}
E_{\alpha, 1}(-x) \sim \frac{1}{1+x}, \quad x \in \R_+ .
\end{equation}

The asymptotic behavior~\eqref{ML_estimate} follows from an integral 
representation
\[
E_{\alpha,\beta}(z)=\frac{1}{2\pi 
i}\int_{\mathcal{C}}\frac{t^{\alpha-\beta}e^t}{t^\alpha-z}dt,
\]
where \(\mathcal{C}\) is an infinite contour in the complex plane. For details 
we refer to~\cite[Chapter 18]{Erdelyi53}. Alternatively, one  can use the 
connection to the Fox \(H\)-function and use the asymptotic behavior known for 
the Fox \(H\) functions, see Section~\ref{section_Fox_asymptotics}. For the function \(Y\) we obtain the asymptotics
\begin{equation}\label{Y_FT_estimate}
\widehat{Y}(t, \xi) \sim \frac{t^{\alpha-1}}{1+|\xi|^{2\beta} t^{2\alpha}}.
\end{equation}

\subsection{Main results}

Our first theorem states that a classical solution of (\ref{classical_equation})  has an integral representation involving \(Z\) 
and \(Y\) mentioned above. Since we defined the fractional Laplacian via the Fourier transform, we need to 
guarantee that \(\mathcal{F}^{-1}_{\xi\to x}(|\xi|^\beta\widehat{u}(t,\xi))\) 
determines a continuous function in \(x\). In particular, by the 
Riemann-Lebesgue Lemma, 
this is true if \(|\cdot|^\beta\widehat{u}(t,\cdot)\in L^1(\Rn)\). We also need 
that \(u(\cdot,x)\) is a continuous function up to \(0\) for all \(x\in\Rn\).
For these purposes, we impose the condition
\begin{equation}\label{f_hat_assumption}
|\widehat{f}(t,\xi)| \le C|g(\xi)|
\end{equation}
for the forcing term \(f\), where the function $g$ satisfies  
\begin{equation}\label{g_assumption}
(1+|\cdot|^\beta) g(\cdot) \in L^1(\Rn),
\end{equation}
and $C>0$ is a constant which is uniform in time.

\begin{theorem}\label{fundamental_solution_thm}
Let $u_0 \in L^1(\Rn)$ be such a function that $\widehat u_0 
\in L^1(\Rn)$ and let $f$ be a function satisfying 
\(f(t,\cdot)\in L^1(\Rn)\) for all \(t\ge 0\) and~\eqref{f_hat_assumption} with 
\(g\) satisfying~\eqref{g_assumption}. Define 
\begin{equation}\label{fund_solution}
Z(t,x)=\pi^{-d/2}|x|^{-d}H_{32}^{12}\bigl[ 2^\beta 
t^{\alpha}|x|^{-\beta} \big| \begin{smallmatrix} (1-\frac{d}{2},\frac\beta2), 
&(0,1), 
&(0,\frac\beta2)\\ (0,1),& (0,\alpha )& 
\end{smallmatrix} \big]
\end{equation}
and 
\begin{equation}\label{fund_solution2}
Y(t,x)=\pi^{-d/2}t^{\alpha-1}|x|^{-d}H_{32}^{12}\bigl[ 2^\beta 
t^{\alpha}|x|^{-\beta} \big| \begin{smallmatrix} (1-\frac{d}{2},\frac\beta2), 
&(0,1), 
&(0,\frac\beta2)\\ (0,1),& (1-\alpha,\alpha )& 
\end{smallmatrix} \big].
\end{equation}
Then the function
\[
\Psi(t,x)=\int_{\Rn} Z(t,x-y)u_0(y)\, dy+\int_0^t\int_{\Rn} Y(t-s,x-y)f(s,y)\, dy \, ds
\]
is a classical solution to problem~\eqref{classical_equation}.
\end{theorem}

\begin{remark}
In our asymptotic analysis we prefer to use the similarity variable 
\(R=t^{-\alpha}|x|^\beta\) similarly as in~\cite{EideKoch04}. Therefore, it is 
desirable to use the property \((ii)\) of Lemma~\ref{Fox_properties} and 
write \(Z\) in a form
\begin{equation}\label{fund_solution_Z2}
Z(t,x)=\pi^{-d/2}|x|^{-d}H_{23}^{21}\bigl[ 2^{-\beta} 
t^{-\alpha}|x|^{\beta} \big| \begin{smallmatrix} (1,1),& 
(1,\alpha)& \\ (\frac{d}{2},\beta/2), 
&(1,1), 
&(1,\beta/2) 
\end{smallmatrix} \big]
\end{equation}
and \(Y\) in a form
\begin{equation}\label{fund_solution_Y2}
Y(t,x)=\pi^{-d/2}t^{\alpha-1}|x|^{-d}H_{23}^{21}\bigl[ 2^{-\beta} 
t^{-\alpha}|x|^{\beta} \big| \begin{smallmatrix} (1,1),& (\alpha,\alpha )& \\ 
(\frac{d}{2},\beta/2), 
&(1,1), 
&(1,\beta/2)
\end{smallmatrix} \big].
\end{equation}

Observe that in the special case \(\beta=2\), we obtain the time-fractional diffusion equation. Its decay properties have been studied in~\cite{KempSiljVergZach14} and for the behavior of its fundamental solution, we refer to~\cite{Koch90}. If we restrict our formula~\eqref{fund_solution} to the case $\beta=2$, it reduces to
\begin{equation}\label{fund_solution_beta2}
Z(x,t)=\pi^{-d/2}\abs{x}^{-d}H_{32}^{12}\bigl[ 4 
t^{\alpha}|x|^{-2} \big| \begin{smallmatrix} (1-\frac{d}{2},1), &(0,1), 
&(0,1)\\ (0,1),& (0,\alpha )& 
\end{smallmatrix} \big].
\end{equation}
Using the properties (ii) and (iii) of the Fox $H$-function from Lemma~\ref{Fox_properties} gives
\[
\begin{split}
H_{32}^{12}\bigl[ 4 
t^{\alpha}r^{-2} \big| \begin{smallmatrix} (1-\frac{d}{2},1), &(0,1), 
&(0,1)\\ (0,1),& (0,\alpha )& 
\end{smallmatrix} \big]&=H_{21}^{02}\big[ 4 
t^{\alpha}r^{-2} \big| \begin{smallmatrix}
            (1-\frac{d}{2},1), &(0,1)\\
(0,\alpha ) &           
                       \end{smallmatrix} \big]\\
&=H_{12}^{20}\big[ \frac{1}{4}|x|^2t^{-\alpha} \big| \begin{smallmatrix}
       (1,\alpha) & \\
(\frac{d}{2},1),&(1,1)                                              
                                                     \end{smallmatrix}
\big].
\end{split}
\]
Therefore the formula~\eqref{fund_solution_beta2} reads as
\[
Z(t,x)=\pi^{-d/2}|x|^{-d} H_{12}^{20}\big[ \frac{1}{4}|x|^2t^{-\alpha} \big| 
\begin{smallmatrix}
       (1,\alpha) & \\
(\frac{d}{2},1),&(1,1)                                              
                                                     \end{smallmatrix}
\big] , 
\]
which is exactly the same as obtained by Kochubei in \cite[Formula 
(18)]{Koch90}.

As explained earlier, the functions \(Z\) and \(Y\) can be derived by taking the Fourier 
transform with respect to the spatial variable \(x\) and the Laplace transform with respect to
time in~\eqref{prob}. For more details we refer to~\cite{Duan05}. Our contribution is in showing that they induce a representation formula, even for relatively rough initial and forcing data.
\end{remark}

Adopting the notion of the Green matrix from~\cite{EideKoch04}, we call the 
pair 
$(Z, Y)$ the {\em matrix of fundamental solutions} of equation~\eqref{classical_equation}. 
Next we define the concept of mild solutions by means of the above 
representation formula.

\begin{definition}
Let $u_0$ and $f$ be Lebesgue measurable functions on $\R^d$ and $[0,\infty)\times \R^d$, respectively. The function $u$ defined by
\begin{align*}
u(t,x)&=\int_{\Rn} Z(t,x-y)u_0(y)\, dy+\int_0^t\int_{\Rn} Y(t-s,x-y)f(s,y)\, dy 
\, ds \\
&=:u_{init}(t,x) + u_{forc}(t,x)
\end{align*}
is called the {\em mild solution}  of the Cauchy 
problem~\eqref{classical_equation} whenever the integrals in the above formula 
are well defined.
\end{definition}

We are particularly interested in the case where the data belong to some 
Lebesgue spaces. Note that our case differs from the usual heat equation. 
For example, in the case of the heat equation it is enough that \(u_0\in 
C(\Rn)\cap L^\infty(\Rn)\) for the above defined \(u\) to be the classical solution 
of the homogeneous equation. As we shall see, for $d \ge 2$ and $\alpha<1$ the function \(Z(t,x)\) has a singularity not 
only in \(t\), but also in \(x\), which implies that more smoothness on 
\(u_0\) 
is required. The function \(Y\) also has a singularity both in \(t\) and \(x\). 
Notice that 
this resembles the Laplace equation, for which the 
fundamental solution \(u(x)=c(d)|x|^{2-d}\) has a singularity at \(x=0\). In a 
sense this reflects the {\em elliptic nature} of the nonlocal PDE when $\alpha<1$.


Next we turn in to the decay of mild solutions.  
We give a quantitative rate at which the solution decays to its fundamental solution and, moreover, if the first moment of the initial datum is finite, we can say even more. These results 
are analogous with the ones for the heat equation in~\cite{Zuaz03}. However, unlike in the case of caloric functions, we need to restrict our study of the $L^p$-decay to a certain range of possible values of $p$. This is caused 
by the fact that the fundamental solution lacks integrability for large enough 
\(p\). Note that this does not happen for the heat kernel, which belongs to 
\(L^\infty(\Rn)\) for all $t>0$. In the limiting case we prove a convergence result in the 
weak \(L^p\)-norm.

Denote 
\[
\kappa_1(\beta, d)=
\begin{cases}
\frac{d}{d-\beta+1}, \quad \text{for} \ \ d > \beta-1, \\
\infty, \quad \text{for} \ \ d \le \beta-1
\end{cases}
\]
and
\[
\kappa_2(\beta, d)=\begin{cases} \frac{d}{d-2\beta}, \quad\text{if} \ \ d > 2\beta, \\
\infty, \quad\text{otherwise}. \end{cases}
\] 
In order to obtain decay for the solution, we need to assume that there exists a $\gamma >1$ such that
\begin{equation}\label{decay_cond_f}
\|f(t,\cdot)\|_{L^1(\Rn)}\lesssim (1+t)^{-\gamma},\quad t> 0.
\end{equation}
Set also
\[
M_{init}=\int_{\Rn} u_0(y)\,dy\quad\text{and}\quad M_{forc}=\int_{0}^\infty\int_{\Rn} f(t,y)\,dy \, dt.
\] 

With this notation we have the following result.

\begin{theorem} \label{zuazua}
Let $d \ge 1$, $u_0\in L^1(\Rn)$ and $f \in L^1(\R_+\times \Rn)$.  Suppose $f$ satisfies~\eqref{decay_cond_f} with some
$\gamma>1$. Assume that \(u\) is the mild solution of equation~\eqref{classical_equation}.

\noindent (i) Then
\[
t^{\frac{\alpha d }{\beta}\,\left(1-\frac{1}{p}\right)}\|u_{init}(t, \cdot)-M_{init}Z(t, \cdot)\|_{L^p} \rightarrow 0,\quad\mbox{as}\;t\rightarrow \infty,
\]
\ \quad for all $p \in [1, \kappa_1)$, and
\[
t^{1+\frac{\alpha 
d}{\beta}(1-\frac{1}{p})-\alpha}\|u_{forc}(t,\cdot)-M_{forc}Y(t,\cdot)\|_{L^p} \to 0,\quad 
t\to\infty,
\]
\ \quad for all \(1\le p\le\infty\), if \(\alpha=1\) or \(d<2\beta\), and for \(p \in [1, \kappa_2)\), if \(d \ge 2\beta\).

(ii) Assume in addition that $\||x|u_0\|_{L^1}<\infty$. Then
\[
t^{\frac{\alpha d }{\beta}\,\left(1-\frac{1}{p}\right)}\|u_{init}(t, \cdot)-M_{init}Z(t, \cdot)\|_{L^p}\lesssim t^{ -\frac{\alpha}{\beta}},\quad t>0.
\]
Moreover, in the limit case $p=\kappa_1(\beta, d)$ we have
\[
t^{\frac{\alpha(\beta-1) }{\beta}}\|u_{init}(t, \cdot)-M_{init}Z(t, \cdot)\|_{L^{\kappa_1(\beta, d),\,\infty}}\lesssim t^{ -\frac{\alpha}{\beta}},\quad t>0.
\]
\end{theorem}

Continuing on decay results, we now turn to study the $L^2$-decay of mild solutions. 
Observe the {\em critical dimension phenomenon} that the decay rate does not improve when the dimension is increased after $d > 2\beta$. Thus, the non-local case is markedly different from that of the standard caloric functions. Importantly, in Section~\ref{L2} we will also show the decay rate provided here is optimal. In particular, the decay rate below is sharp for all initial data $u_0$ such that $\int_{\Rn} u_0 \, dx \neq 0$.

\begin{theorem}\label{thm_optimal_decay}
Let $\alpha\in (0,1)$, \(d\ge 1\) and \(d\neq 2\beta\). Suppose $u$ is the mild solution of the Cauchy problem~\eqref{classical_equation} with \(u_0\in L^1(\R^d)\cap 
L^2(\R^d)\) and $f\equiv 0$. Then
\[
\|u(t,\cdot)\|_{L^2} \lesssim t^{-\alpha\mathrm{min}\{1,\frac{d}{2\beta}\}},\quad t>0.
\]
Moreover, in case $d=2\beta$ we have
\[
\|u(t,\cdot)\|_{L^{2, \infty}} \lesssim t^{-\alpha},\quad t>0.
\]

\end{theorem}

%
%

Finally, in the following theorem we turn in to the decay of weak solutions. The proof is based on a comparison principle and a priori estimates. It is an open question whether the decay rate here is optimal as it is not as good as the one obtained by the Fourier methods in the previous theorem. The same phenomenon is present already in the case of the time fractional diffusion~\cite{KempSiljVergZach14}. Observe that our method gives the correct decay when applied to the heat equation. 

For $s\in (0,1)$ we set
\[
[v]_{W^{s,1}(\R^d)}=\int_{\R^d} \int_{\R^d} \frac{|v(x)-v(y)|}{|x-y|^{d+s}}\,dx\,dy,
\]
which is the Gagliardo-seminorm of the Sobolev Slobodecki space $W^{s,1}(\R^d)$. 

\begin{theorem}\label{decay_weak_solution}
Let $u_0 \in L^1(\Rn) \cap L^2(\Rn)$ and suppose the kernel $K$ satisfies~\eqref{K_assumption} with some $\beta\in (0,2)$. Let $u$ be the weak solution of equation~\eqref{weak_equation} with initial condition $u|_{t=0}=u_0$, and assume that
\begin{equation} \label{ucondition}
\int_0^T [u(t,\cdot)]_{W^{\frac{\beta}{2},1}(\R^d)}\,dt<\infty\quad \mbox{for all}\,\;T>0.
\end{equation}
Then 
\[
\|u(t, \cdot) \|_{L^2} \lesssim (1+t)^{-\frac{\alpha d}{d+2\beta}}, \quad t>0.
\]
\end{theorem}

\begin{remark}
(i) As our proof shows, Theorem \ref{decay_weak_solution} (trivially) extends to the case where the kernel $K$ also depends
on time $t$, that is $K=K(t,x,y)$, provided that $K$ is measurable on $(0,\infty)\times \R^d \times \R^d$ and
\eqref{K_assumption} holds a.e.\ in this set with $K(t,x,y)$ in place of $K(x,y)$. In this more general formulation,
our result can be also applied to certain {\em quasilinear} equations which satisfy suitable structure conditions. 

(ii) The authors believe that by careful estimates for appropriate approximating equations (as in \cite{Zach09})
combined with Gagliardo-Nirenberg inequalities one can show that the weak solution of equation~\eqref{weak_equation} always satisfies the technical condition (\ref{ucondition}) provided that $u_0 \in L^1(\Rn) \cap L^2(\Rn)$. For the sake of simplicity we do not go into the details here.    
\end{remark}

\section{Auxiliary tools}

We recall some classical results which are needed in the theory.

\subsection{Review of harmonic analysis}

Let $f\star g$ denote the convolution of $f,g$ on $\Rn$. We recall the Young's inequality for convolutions: for any triple $1\le p,q,r\le \infty$ satisfying $1+\frac{1}{r}=\frac{1}{p}+\frac{1}{q}$
\begin{equation} \label{Young1}
\|f\star g\|_{L^r}\le \|f\|_{L^p} \|g\|_{L^q},\quad f\in L^p(\Rn),\,g\in L^q(\Rn).
\end{equation}

We also recall the strengthened version for weak type spaces: Let 
$1<p,q,r<\infty$
satisfy $1+\frac{1}{r}=\frac{1}{p}+\frac{1}{q}$. Then
\begin{equation} \label{Young2}
\|f\star g\|_{L^r}\le C(p,q,r)\|f\|_{L^{p,\infty}} \|g\|_{L^q},\quad f\in 
L^{p,\,\infty}(\Rn),\,g\in L^q(\Rn),
\end{equation}
see \cite[Theorem 1.4.24]{Graf04}. In the case $q=1$ there also holds
\begin{equation} \label{Youngw}
\|f\star g\|_{L^{p,\,\infty}}\le C(p) \|f\|_{L^{p,\,\infty}} \|g\|_{L^1},\quad f\in 
L^p(\Rn),\,g\in L^1(\Rn),
\end{equation}
for all $1<p<\infty$, see \cite[Theorem 1.2.13]{Graf04}.

For the nonhomogeneous problem we need the integral form of the Minkowsky 
inequality in the following form. Let \(1\le p<\infty\) and \(F\) be a 
measurable function on the product space \(\mathbb{R}_+\times\Rn\). Then
\[
\Big(\int_{\mathbb{R}_+}\Big(\int_{\Rn}|F(t,x)|\,dx\Big)^p 
dt\Big)^{\frac{1}{p}}\le \int_{\Rn}\Big(\int_{\mathbb{R}_+}|F(t,x)|^p\, 
dt\Big)^{\frac{1}{p}} dx.
\]

We will also need the following decomposition lemma from \cite{DuoaZuaz92}.
\begin{lemma} \label{decomp}
Suppose $f\in L^1(\Rn)$ such that $\int_{\Rn} |x|\,|f(x)|\,dx<\infty$. Then there exists $F\in L^1(\Rn;\Rn)$ such that
\[
f=\left(\int_{\Rn} f(x)\,dx\right)\,\delta_0+\mbox{{\em div}}\,F
\]
in the distributional sense and
\[
\|F\|_{L^1(\Rn;\Rn)}\le C_d \int_{\Rn} |x|\,|f(x)|\,dx.
\]
\end{lemma}

%

We will also need the boundedness of the Riesz 
potential 
\[
(-\Delta)^{-\frac{\beta}{2}}f:= c_{d, \beta}\int_{\Rn} \frac{f(y)}{|x-y|^{d-\beta}} \, dy,
\]
for \(0<\beta<d\). We have the {\em 
Hardy-Littlewood-Sobolev theorem on fractional 
integration}~\cite[Theorem 6.1.3]{Graf04}:

\begin{theorem}\label{HLS_thm}
Let \(1\le p<d/\beta\) and $f \in L^p(\Rn)$. Then
\[
\|(-\Delta)^{-\frac{\beta}{2}}f\|_{L^q(\Rn)}\lesssim \|f\|_{L^p(\Rn)}
\]
for  \(p>1\) and
$$q=\frac{dp}{d-p\beta}.$$
 In case \(p=1\), we have
\[
\|(-\Delta)^{-\frac{\beta}{2}}f\|_{L^{\frac{dp}{d-p\beta},\infty}(\Rn)}\lesssim \|f\|_{L^1(\Rn)}.
\]
\end{theorem}

\subsection{Asymptotic behavior of the Fox 
$H$-functions}\label{section_Fox_asymptotics}
When developing the asymptotic behavior of the fundamental solution, we need 
the following representation formulas for the Fox $H$-function $H_{23}^{21}$. 
Here we have omitted the parameters of the Fox $H$-function and $H_{23}^{21}$ 
refers to the Fox $H$-function appearing either in~\ref{fund_solution_Z2} or 
in~\ref{fund_solution_Y2}. The following results hold for both functions.

\begin{theorem}\label{fox_series_expansion}
Let either \(\beta>\alpha\) and \(z\neq 0\), or \(\alpha=\beta\) and \(0<|z|<\delta\) with 
\(\delta=\alpha^{-\alpha}(\frac{1}{2})^{1/2}(\frac{\beta}{2})^{\beta/2}\). Then the Fox H-function \(H_{23}^{21}(z)\) is an
analytic function of \(z\) and
\begin{equation}\label{Fox_series}
H^{21}_{23}(z)=\sum_{j=1}^2\sum_{l=0}^\infty 
\mathrm{Res}_{s=b_{jl}}\big[\mathcal{H}^{21}_{23}(s)z^{-s}\big],
\end{equation}
where \(b_{jl}\) are given in~\eqref{poles_b_jl}. 
\end{theorem}

The asymptotic behavior of \(H^{21}_{23}(z)\) as \(z\to 0\) follows 
immediately from~\eqref{Fox_series} in the case \(\beta\ge\alpha\) by 
calculating the residues. If 
$0< \alpha < \beta$ and \(|\mathrm{arg}z|<\pi(1-\frac{\alpha}{2})\), then
\begin{equation}\label{Fox_series_ab}
H^{21}_{23}(z)\sim -\sum_{j=1}^2\sum_{l=0}^\infty 
\mathrm{Res}_{s=-b_{jl}}\big[\mathcal{H}^{21}_{23}(-s)z^{s}\big],
\end{equation}
when $z \to 0$. Again, the asymptotic behavior follows immediately by 
calculating the residues. 

The asymptotic behavior at infinity is more 
complicated to derive. For details we refer to~\cite{Braa36} and 
~\cite[Sections 1.3 and 1.5]{KilbSaig04}.

\begin{theorem}\label{Braa36}
The asymptotic expansion at infinity of the Fox H-function \(H_{23}^{21}(z)\) 
has the form
\begin{equation}\label{Fox_series_infinity}
H_{23}^{21}(z)\sim\sum_{k=0}^\infty h_{k}z^{-k},
\end{equation}
where the constants \(h_k\) have the form
\begin{equation}\label{h_k}
\begin{split}
h_k&=\lim_{s\to a_{1k}}\big[-(s-a_{1k})\mathcal{H}^{21}_{23}(s)\big]\\
&=\frac{(-1)^k}{k!\alpha_1}\frac{\Gamma(b_1+(1-a_1+k)\frac{\beta_1}{\alpha_1}
)\Gamma(b_2+(1-a_1+k)\frac{\beta_2}{\alpha_1})}{\Gamma(a_2+(1-a_1+k)\frac{
\alpha_2}{\alpha_1})\Gamma(1-b_3-(1-a_1+k)\frac{\beta_3}{\alpha_1})}
\end{split}
\end{equation}
in view of the relation
\[
\mathrm{Res}_{s=a_{1k}}\big[\mathcal{H}^{21}_{23}(s)z^{-s}\big]=h_kz^{-a_{1k}}
=h_kz^{(a_1-1-k)/\alpha_1}.
\]
\end{theorem}

\section{Behavior of the fundamental solutions}

We start by showing some basic properties of the fundamental solutions $Z$ and 
$Y$. The first lemma provides an important connection between the functions $Z$ and 
$Y$. Note, in particular, that \(Z\) and \(Y\) are identical in the case 
\(\alpha=1\). 

\begin{lemma}\label{ZY_connection}
The fundamental solutions \(Z\) and \(Y\) of equation~\eqref{prob} are 
connected via $Y=\partial_t^{1-\alpha}Z$.

\begin{proof}
Observe first that
\[
\partial_t^{1-\alpha}f(at)=a^{1-\alpha}(\partial_t^{1-\alpha}f)(at)
\]
for a sufficiently smooth function $f$ and for a constant $a \in \R_+$.

Now we combine this with Lemma~\ref{Fox_properties} (iv) to obtain
\begin{align*}
&\partial_t^{1-\alpha} H_{32}^{12}\bigl[ 2^\beta 
t^{\alpha}|x|^{-\beta} \big| \begin{smallmatrix} (1-\frac{d}{2},\beta/2), 
&(0,1), 
&(0,\beta/2)\\ (0,1),& (0,\alpha )& 
\end{smallmatrix} \big] \\
&=
t^{\alpha-1}H_{43}^{13}\bigl[ 2^\beta 
t^{\alpha}|x|^{-\beta} \big| \begin{smallmatrix} (0,\alpha), &(1-\frac{d}{2},\beta/2), 
&(0,1), 
&(0,\beta/2)\\ (0,1),& (0,\alpha ), &(1-\alpha, \alpha)
\end{smallmatrix} \big].
\end{align*}
We need to study the Mellin transform of the above Fox $H$-function. That is
\begin{align*}
\mathcal{H}_{43}^{13}(s)&=\frac{\Gamma(s)\Gamma(1-\alpha s)\Gamma\l(\frac{d}{2}-\frac{\beta}{2}s\r)\Gamma(1-s)}{\Gamma\l(\frac{\beta}{2}s\r)\Gamma(1-\alpha s)\Gamma(\alpha - \alpha s)} \\
&=\frac{\Gamma(s)\Gamma\l(\frac{d}{2}-\frac{\beta}{2}s\r)\Gamma(1-s)}{\Gamma\l(\frac{\beta}{2}s\r)\Gamma(\alpha - \alpha s)}=\mathcal{H}_{32}^{12}(s)
\end{align*}
We obtain
\begin{align*}
Y(t,x)&=\pi^{-d/2}\abs{x}^{-d}t^{\alpha-1}H_{32}^{12}\bigl[ 2^\beta 
t^{\alpha}|x|^{-\beta} \big| \begin{smallmatrix} (1-\frac{d}{2},\beta/2), 
&(0,1), 
&(0,\beta/2)\\ (0,1),& (1-\alpha,\alpha )& 
\end{smallmatrix} \big] \\
&=\partial_t^{1-\alpha} Z(t,x),
\end{align*}
as required.
\end{proof}
\end{lemma}

Before moving into providing the exact behavior of the fundamental solutions $Z$ and $Y$, we give the following remark.

\begin{remark}\label{nonnegative}
Observe that the functions $Z$ and $Y$ are both non-negative and, moreover, $Z$ induces a probability measure. 

Indeed, by Bochner's Theorem the non-negativity follows from showing that the Fourier transforms 
\(\widehat{Z}(t,\cdot)\) and \(\widehat{Y}(t,\cdot)\) are positive definite on \(\Rn\) \cite{Boch12}. Recalling that \(\widehat{Z}(t,\cdot)\) 
and \(\widehat{Y}(t,\cdot)\) can be represented in terms of the 
Mittag-Leffler functions \(E_{\alpha,1}\) and \(E_{\alpha,\alpha}\), for the positive definiteness it is enough to show that the functions \(f(r)=E_{\alpha,1}(-t^\alpha r^{\frac{\beta}{2}})\) and 
\(g(r)=E_{\alpha,\alpha}(-t^{\alpha} r^{\frac{\beta}{2}})\) are completely 
monotone on \(\mathbb{R}_+\)~\cite[Theorem 3]{Scho38}. But since the functions \(x\mapsto 
E_{\alpha,1}(-x)\), \(x\mapsto E_{\alpha,\alpha}(-x)\) and \(x\mapsto 
cx^{\frac{\beta}{2}-1}\) with \(c\ge 0\) and \(\beta\le 2\) are known to be 
completely monotone on \(\mathbb{R}_+\)~\cite{MillSamk01}, we obtain the result. 

Finally, by~\eqref{fund_fourier_trans} we have
\[
\int_{\Rn} Z(t,x) \, dx = \widehat{Z}(t,0)= E_{\alpha,1}(0)=1,
\]
for every $t>0$, which yields that $Z(t, \cdot) \ge 0$ induces a probability measure on $\Rn$. 

\end{remark}

When proving the decay estimates we will need the following asymptotic estimates 
for the fundamental solutions. We begin by studying the function \(Z\).

\begin{lemma}\label{fund_sol_asympt}

Let $d \in \Z_+$, $0<\alpha \le 1$ and $0< \beta \le 2$. Denote \(R:=|x|^\beta 
t^{-\alpha}\). Then the function $Z$ has the following asymptotic behavior:
\begin{itemize}

\item[(i)] If \(R\leq 1\), then
\[
Z(t,x)\sim \begin{cases}
             t^{-\alpha d/\beta}, &\textrm{if $\alpha=1$, or $\beta > d$ and 
$0<\alpha<1$},\\
t^{-\alpha}(|\log(|x|^\beta t^{-\alpha})|+1), &\textrm{if $\beta=d$ and  
$0<\alpha <1$},\\
t^{-\alpha}|x|^{-d+\beta} &\textrm{if $0<\beta < d$ and $0<\alpha<1$}.
            \end{cases}
\]

\item[(ii)] If \(R\ge 1\), then 
\[
Z(t,x)\sim t^{\alpha}|x|^{-d-\beta},\quad \textrm{if \(\beta<2\)}.
\]
In the special case $\beta=2$ there holds
\[
Z(t,x)\lesssim t^{\alpha}|x|^{-d-2}.
\]

\end{itemize}

\begin{proof} {\em \((i)\) \(R\le 1\)}: We 
start with the case 
\(0<\alpha<1\). Since the asymptotic behavior depends on whether 
\(\beta\ge\alpha\) or \(\beta<\alpha\), we have study different subcases. First 
of all, recall the definition of $Z$ as
\[
Z(t,x)=\pi^{-d/2}|x|^{-d}H_{23}^{21}\bigl[ 2^{-\beta} 
t^{-\alpha}|x|^{\beta} \big| \begin{smallmatrix} (1,1),& (1,\alpha )&  \\
(\frac{d}{2},\frac{\beta}{2}), 
&(1,1), 
&(1,\frac{\beta}{2})
\end{smallmatrix} \big].
\]
In order to figure out the asymptotic behavior of $Z$, we need to study the 
above Fox $H$-function. As it was mentioned in 
Section~\ref{section_Fox_asymptotics}, the asymptotic behavior follows by 
calculating the residues. We provide the details for the reader's convenience.
 
\noindent{\em The subcase \(\beta\ge \alpha\)}: We have
\begin{equation}\label{H_expansion}
H^{21}_{23}(z)=\sum_{j=1}^2\sum_{l=0}^\infty 
\mathrm{Res}_{s=b_{jl}}\big[\mathcal{H}^{21}_{23}(s)z^{-s}\big]
\end{equation}
by Theorem~\ref{fox_series_expansion}. Recall the definition of the 
Mellin transform
\begin{align*}
\mathcal{H}^{21}_{23}(s):=& \mathcal{M}(H^{21}_{23}\big[ z \big| 
\begin{smallmatrix}
       (1,1),&(1,\alpha) &\\
(\frac{d}{2},\frac{\beta}{2}), &(1,1), &(1,\frac{\beta}{2})
      \end{smallmatrix}
\big])(s)\\
=&\frac{\Gamma(\frac{d}{2}+\frac{\beta}{2}s)\Gamma(1+s)\Gamma(-s)}{
\Gamma(1+\alpha s)\Gamma(-\frac{\beta}{2}s)}.
\end{align*}
In light of~\eqref{H_expansion}, the asymptotic behavior is determined by the 
largest value of \(s\), which is a pole of $\mathcal{H}^{21}_{23}(s)$. Now, for 
$0<\alpha <1$ the above Mellin transform has poles at $s=-1$ and 
$s=-d/\beta$. Suppose first that $\beta > d$. The only value for \(d\) to this 
happen is \(d=1\), when \(1<\beta\le 2\), whereas for \(0<\beta\le 1\) there 
is no such \(d\). Then the asymptotics is determined 
by the pole at 
$-1/\beta$ and the behavior of the  $H_{21}^{23}$ near zero is $H_{21}^{23}(z) 
\sim z^{1/\beta}$. This yields
\[
Z(t,x) \sim t^{-\alpha /\beta},
\]
as required. 

Assume next that $\alpha\le\beta < d$. Then the largest value of \(s\) 
such that the Mellin transform has a pole is $s=-1$ and we obtain 
$H_{21}^{23}(z) \sim z$. This produces
\[
Z(t,x) \sim t^{-\alpha}|x|^{-d+\beta}.
\]

In the case $\beta=d$ the Mellin transform has a second order pole at 
$s=-1$. Then the residue can be calculated as
\begin{align*}
\mathrm{Res}_{s=-1}\big[\mathcal{H}^{21}_{23}(s)z^{-s}\big] &=\lim_{s\to -1} \frac{d}{ds}[(s+1)^2\mathcal{H}^{21}_{23}(s)z^{-s}] \\
&=z \lim_{s\to -1} 
\frac{d}{ds}[(1+s)^2\mathcal{H}^{21}_{23}(s)]\\
&\quad+\lim_{s\to 
-1}[(1+s)^2\mathcal{H}^{21}_{23}(s) \frac{d}{ds}(z^{-s})].
\end{align*}
Since \((1+s)\Gamma(1+s)=\Gamma(2+s)\) and \( 
\frac{d}{2}(1+s)\Gamma(\frac{d}{2}+\frac{\beta}{2}s)=\Gamma(\frac{d}{2}+1+\frac
{\beta}{2}s)\) are analytic at \(s=-1\), the limits
\[
\lim_{s\to -1} 
\frac{d}{ds}[(1+s)^2\mathcal{H}^{21}_{23}(s)]\quad\textrm{and}\quad\lim_{s\to 
-1}(1+s)^2\mathcal{H}_{23}^{21}(s)
\]
exist. Moreover, since
\[
\lim_{s\to -1} \frac{d}{ds}(z^{-s}) = z \log z,
\]
we may conclude that in this case $H^{21}_{23}(z) \sim z \log z$ and thus
\[
Z(t,x) \sim t^{-\alpha}(|\log(|x| t^{-\alpha})|+1),
\]  
again as required. 

{\em The subcase \(\beta<\alpha\)}: Since we are interested on the 
asymptotics of the Fox \(H\)-function for \(z\in\R_+\), the asymptotics is 
given by~\eqref{Fox_series_ab}. Because $0<\beta <\alpha \le 1 \le d$, we 
have $d/\beta > 1$ and the leading term is determined by
\[
\mathrm{Res}_{s=1}[\mathcal{H}_{23}^{21}(-s)z^{s}].
\]
Therefore
\[
Z(t,x) \sim t^{-\alpha}|x|^{-d+\beta}
\]

In the special case \(\alpha=1\) we see that the Mellin transform of 
\(H^{21}_{23}(z)\) reduces to
\[
\mathcal{H}^{21}_{23}(s)=\frac{\Gamma(\frac{d}{2}+\frac{\beta}{2}s)\Gamma(-s)}{
\Gamma(-\frac{\beta}{2})}.
\]

Therefore the asymptotics is given by the pole at 
\(s=-\frac{d}{\beta}=-\frac{\alpha d}{\beta}\). 
Proceeding as above we end up to the desired estimate.

{\em \((ii)\) \(R\ge 1\)}: We use the 
asymptotic behavior of the Fox 
$H$-functions provided by Theorem~\ref{Braa36}:
\[
H_{23}^{21}(z)\sim\sum_{k=0}^\infty h_{k}z^{-k},
\]
for constants $h_k$ defined in~\eqref{h_k}. We aim to find the smallest value of 
$k$ such that $h_k \neq 0$. Let's first study the case $0 < \beta <2$. Now
\begin{equation}\label{h_k_calc1}
h_{0}=\frac{\Gamma(\frac{d}{2})\Gamma(1)}{\Gamma(1)\Gamma(0)}=0\quad\text{and}
\quad
h_{1}=-\frac{\Gamma(\frac{d}{2}+\frac{\beta}{2})\Gamma(2)}{
\Gamma(1+\alpha)\Gamma(-\frac
{\beta}{2})}\neq 0.
\end{equation}
Therefore the leading term in the expansion~\eqref{Fox_series_infinity} is 
\(h_1 z^{-1}\) so
\[
H^{21}_{23}(z)\sim z^{-1},\quad z\to\infty
\]
and we obtain the claim of the lemma.

If \(\beta=2\), we see from~\eqref{h_k_calc1} that \(h_0=0\) and \(h(1)=0\), 
since 
\[
\Gamma(-\frac{\beta}{2})=\Gamma(-1)=\infty. 
\]
Therefore the claim is 
true also in this case. However, we can continue to deduce \(h_k=0\) for all 
\(k\in\Z_+\). One can prove that now actually \(Z(t,x)\) decays 
in terms of \(R\to\infty\), but we do not need that fact in our 
considerations.
\end{proof}
\end{lemma}

The next lemma gives the behavior of the fundamental solution $Y$. The proof 
is similar to the previous lemma.

\begin{lemma}\label{fund_sol_asympt_Y}
Let $d \ge 1$, $0<\alpha \le 1$ and $0<\beta \le 2$. Denote \(R:=|x|^\beta 
t^{-\alpha}\). Then the function $Y$ has 
the following asymptotic behavior:
\begin{itemize}
\item[(i)] If \(R\leq 1\), then
\[
Y(t,x)\sim \begin{cases}
  t^{-\alpha-1}           |x|^{-d+2\beta} , &\textrm{if $d>2\beta$ and 
$0<\alpha<1$},\\
t^{-\alpha-1}|\log(2^{-\beta}|x|^\beta t^{-\alpha})|, 
&\textrm{if $d=2\beta$ and $0<\alpha<1$},\\
t^{\alpha-1-\frac{\alpha d}{\beta}}, &\textrm{if $\alpha=1$, or $d<2\beta$ and 
$0<\alpha<1$}.
            \end{cases}
\]

\item[(ii)] If \(R\ge 1\), then
\[
Y(t,x)\sim t^{2\alpha-1}|x|^{-d-\beta},\quad\textrm{if $0<\beta<2$}.
\]
In the special case $\beta=2$ there holds
\[
Y(t,x)\lesssim t^{2\alpha-1}|x|^{-d-2}.
\]

\end{itemize}
\begin{proof}
The proof is similar to that of Lemma~\ref{fund_sol_asympt}. We omit the 
details. Once again, notice that in the special case \(\beta=2\) the 
function \(Y\) has indeed exponential decay as \(R\to\infty\) but we do not 
need that fact in our calculations.
\end{proof}
\end{lemma}

Next we turn to study the behavior of the derivatives of $Z$ and $Y$.

\begin{lemma}\label{G_gradient}
Let $d \ge 1$, $0<\alpha \le 1$ and $0<\beta \le 2$. Denote \(R:=|x|^\beta 
t^{-\alpha}\). Then the derivatives of the fundamental solution pair $(Z,Y)$ 
have the following asymptotic behavior:

\begin{itemize}

\item[(i)] For the function $Z$ we have
\begin{align*}
&|\nabla Z(t,x)| \sim              t^{-\alpha}|x|^{-d-1+\beta}, \quad \text{if}\ \ R \le 1  \\
&|\nabla Z(t,x)|\sim t^{\alpha}|x|^{-d-1-\beta}, \quad \text{if}\ \ R \ge 1.  
\end{align*}

\item[(i)] For $Y$ we have for $R \le 1$ that
\begin{align*}
|\nabla Y(t,x)| &\sim \begin{cases}
  t^{-\alpha-1}           |x|^{-d-1+2\beta} , &\textrm{if $d+2>2\beta$ and 
$0<\alpha<1$},\\
t^{-\alpha-1}|x\log(2^{-\beta}|x|^\beta t^{-\alpha})|, 
&\textrm{if $d+2=2\beta$ and $0<\alpha<1$},\\
t^{\alpha-1-\frac{\alpha (d+2)}{\beta}}|x|, &\textrm{if $\alpha=1$ or 
$d+2<2\beta$},
            \end{cases} \\
            \text{and}\qquad\qquad& \\
|\nabla Y(t,x)|&\sim t^{2\alpha-1}|x|^{-d-1-\beta}, \quad \text{if} \ \ R \ge 1. 
\end{align*}

\item[(ii)] In addition, for the time derivative of $Y$ we have
\begin{align*}
&|\partial_t Y(t,x)| \sim              t^{-1}Y(t,x),  \quad \text{if}\ \ R \le 1, \\
&|\partial_t Y(t,x)| \sim              t^{2\alpha-2}|x|^{-d-\beta},  \quad \text{if}\ \ R \ge 1.
\end{align*}

\end{itemize}
\begin{proof}
We provide the calculations only for the gradient of the function $Z$. The other cases are handled similarly, but we omit the details. 

Recall the expression for the fundamental solution $Z$:
\[
Z(t,x)=\pi^{-d/2}|x|^{-d}H_{32}^{12}\bigl[ 2^\beta 
t^{\alpha}|x|^{-\beta} \big| \begin{smallmatrix} (1-\frac{d}{2},\beta/2), 
&(0,1), 
&(0,\beta/2)\\ (0,1),& (0,\alpha )& 
\end{smallmatrix} \big].
\]
First of all, we use Lemma~\ref{Fox_properties} (ii) to write the above Fox $H$-function as
\[
H_{32}^{12}\bigl[ 2^\beta 
t^{\alpha}|x|^{-\beta} \big| \begin{smallmatrix} (1-\frac{d}{2},\beta/2), 
&(0,1), 
&(0,\beta/2)\\ (0,1),& (0,\alpha )& 
\end{smallmatrix} \big]
=
 H_{23}^{21}\bigl[ 2^{-\beta} 
t^{-\alpha}|x|^{\beta} \big| \begin{smallmatrix} (1,1), 
&(1, \alpha)\\ 
(\frac{d}{2},\frac{\beta}{2}),& (1,1 ), &(1, \frac{\beta}{2})
\end{smallmatrix} \big].
\]

According to Lemma~\ref{Fox_properties} (i), we have
\[
\frac{d}{d z}H^{21}_{23}\big[ z \big| \begin{smallmatrix}(1,1), 
&(1, \alpha)\\ 
(\frac{d}{2},\frac{\beta}{2}),& (1,1 ), &(1, \frac{\beta}{2})&  
\end{smallmatrix} \big] = z^{-1}H^{22}_{34}\big[ z \big| \begin{smallmatrix} 
(0,1),& (1,1), &(1, \alpha)\\ 
(\frac{d}{2},\frac{\beta}{2}),& (1,1 ), &(1, \frac{\beta}{2}), &(1,1) 
\end{smallmatrix} \big].
\]
Using the product rule for differentiation, we may now calculate
\[
\frac{\partial}{\partial x_j} Z(t, x)= \pi^{-d/2}\frac{x_j}{|x|^{d+2}}[\beta H_{34}^{22}(z)-dH_{23}^{21}(z)].
\]
For simplicity, we have here omitted the set of parameters inside the Fox $H$-functions. Next we analyse the above Fox $H$-functions by studying the corresponding Mellin transforms. We have
\begin{align*}
&\beta \mathcal{H}_{34}^{22}(s)-d\mathcal{H}_{23}^{21}(s) \\
&=\beta \frac{\Gamma(\frac{d}{2}+\frac{\beta}{2}s)\Gamma(1+s)\Gamma(1-s)}{\Gamma(1+\alpha s)\Gamma(-\frac{\beta}{2}s)} -d\frac{\Gamma(\frac{d}{2}+\frac{\beta}{2}s)\Gamma(1+s)\Gamma(-s)}{\Gamma(1+\alpha s)\Gamma(-\frac{\beta}{2}s)} \\
&=-2\l(\frac{d}{2}+\frac{\beta}{2}s\r)\frac{\Gamma(\frac{d}{2}+\frac{\beta}{2}s)\Gamma(1+s)\Gamma(-s)}{\Gamma(1+\alpha s)\Gamma(-\frac{\beta}{2}s)} \\
&=-2\frac{\Gamma(\frac{d+2}{2}+\frac{\beta}{2}s)\Gamma(1+s)\Gamma(-s)}{\Gamma(1+\alpha s)\Gamma(-\frac{\beta}{2}s)} \\
&=-2\mathcal{H}_{23}^{21}\big[ s \big| \begin{smallmatrix}(1,1), 
&(1, \alpha)\\ 
(\frac{d+2}{2},\frac{\beta}{2}),& (1,1 ), &(1, \frac{\beta}{2})&  \end{smallmatrix} \big].
\end{align*}
Thus we obtain
\[
\frac{\partial}{\partial x_j} Z(t, x)= -2\pi^{-d/2}\frac{x_j}{|x|^{d+2}}H_{23}^{21}\big[ 2^{-\beta} 
t^{-\alpha}|x|^{\beta} \big| \begin{smallmatrix}(1,1), 
&(1, \alpha)\\ 
(\frac{d+2}{2},\frac{\beta}{2}),& (1,1 ), &(1, \frac{\beta}{2})&  \end{smallmatrix} \big]
\]
and 
\[
|\nabla Z(t,x)|=2\pi^{-d/2}|x|^{-d-1}\l|H_{23}^{21}\big[ 2^{-\beta} 
t^{-\alpha}|x|^{\beta} \big| \begin{smallmatrix}(1,1), 
&(1, \alpha)\\ 
(\frac{d+2}{2},\frac{\beta}{2}),& (1,1 ), &(1, \frac{\beta}{2})&  \end{smallmatrix} \big]\r|.
\]
Now the result follows from the behavior of the Fox $H$-functions.
\end{proof}
\end{lemma}

\section{Representation formula for solutions}

\subsection{Proof of Theorem~\ref{fundamental_solution_thm}}

We are now ready to prove Theorem~\ref{fundamental_solution_thm}, which 
justifies calling $(Z,Y)$ the matrix of fundamental solutions for the 
equation~\eqref{classical_equation}.

\begin{proof}
 
We need to show that the function
\begin{align*}
\Psi(t,x)&=\int_{\Rn} Z(t,x-y)u_0(y)\, dy+\int_0^t\int_{\Rn} Y(t-s,x-y)f(s,y)\, 
dy \, ds \\
&=:\Psi_1(t,x)+\Psi_2(t,x)
\end{align*}
is a classical solution to equation~\eqref{prob}. We divide the proof into 
three 
steps as there are three requirements in the definition of the classical 
solution.
\\

\noindent {\em Step I:} First we need to prove that $\mathcal{F}^{-1}_{\xi\to 
x}(|\xi|^\beta
\widehat{\Psi}(t, \xi))$ is a continuous function with respect to $x$ for each 
$t>0$. The representation~\eqref{fund_fourier_trans} and the asymptotic behavior of the Mittag-Leffler function given by~\eqref{ML2_estimate} give
\[
|\widehat Z(t,\xi)| \le \frac{C}{1+|\xi|^{\beta}t^\alpha}
\]
for $t  > 0$. Thus $|\xi|^\beta\widehat Z(t,\xi)$ is bounded for all $t>0$ and by using the assumption that $\widehat u_0 \in L^1$ we obtain  
\begin{equation}\label{Z_L1}
|\xi|^\beta\mathcal{F}_{x \to \xi}(Z \star 
u_0)(t,\cdot)=|\xi|^\beta\widehat{Z}(t,\xi)\widehat{u}_0(\xi) \in L^1(\Rn).
\end{equation}

In order 
to estimate $Y$, we use the assumption~\eqref{f_hat_assumption} to obtain
\[
|\widehat{f}(s, \xi)| \lesssim |g(\xi)|
\]
with $|\xi|^\beta g(\xi) \in L^1(\Rn)$. Combining this with~\eqref{Y_FT_estimate} yields
\begin{align*}
|\xi|^\beta\mathcal{F}(Y\, \hat* \, f)(t,\xi) 
&=|\xi|^\beta\int_0^t\widehat{Y}(t-s,\cdot)\widehat{f}(s,\cdot) \, ds \\
&\lesssim 
|\xi|^\beta|g(\xi)|\int_0^t 
\frac{(t-s)^{\alpha-1}}{1+(t-s)^{2\alpha}|\xi|^{2\beta}} \, ds,
\end{align*}
which establishes that $|\cdot|^\beta\mathcal{F}(Y\, \hat*\, f)(t,\cdot) \in L^1$. 
This together with~\eqref{Z_L1} gives, again by the Riemann-Lebesgue lemma, 
that
\[
\mathcal{F}^{-1}_{\xi\to x}(|\xi|^\beta\widehat{\Psi}(t, \xi))
\]
is a continuous function, as required.
\\

\noindent {\em Step II:} 
We proceed as in Section 5.3 of~\cite{EideKoch04}. By 
Lemma~\ref{Fox_properties}, (iv) we have
\[
\partial_t^\alpha H_{32}^{12}[t^\alpha |\begin{smallmatrix}
(1-\frac{d}{2},\frac{\beta}{2}), &(0,1), &(0,\frac{\beta}{2})       
\\
(0,1),&(0,\alpha) &
      \end{smallmatrix}]
    =   t^{-\alpha}H_{43}^{13} [t^\alpha |\begin{smallmatrix}
       (0, \alpha), &(1-\frac{d}{2},\frac{\beta}{2}), &(0,1), 
&(0,\frac{\beta}{2}) \\
(0,1),&(0,\alpha), &(\alpha, \alpha) &
      \end{smallmatrix}].
\]
A detailed study of the asymptotics similarly as in  Lemma~\ref{fund_sol_asympt} can be used to show that \(\partial_t^\alpha  Z(t,\cdot)\) is integrable for all \(t>0\). Thus for all \(x\in\Rn\) the function
$J^{1-\alpha}\Psi_1(t,x)$ is continuously differentiable with respect to time for all $t>0$. We now turn to study $\Psi_2$.


Let $v:=J^{1-\alpha}\Psi_2$. Observe that, after changing the order of 
integration and a change of variables, Lemma~\ref{ZY_connection} gives
\begin{align*}
v(t,x)&=\int_0^t \frac{(t-\tau)^{-\alpha}}{\Gamma(1-\alpha)} \int_0^\tau 
\int_{\Rn} Y(\tau-\lambda, x-y)f(\lambda, y) \, dy \, d\lambda \, d\tau \\
&=\int_0^t  \int_\lambda^t \frac{(t-\tau)^{-\alpha}}{\Gamma(1-\alpha)}\int_{\Rn} 
Y(\tau-\lambda, x-y)f(\lambda, y) \, dy \, d\tau \, d\lambda \\
&= \int_0^t \int_{\Rn} Z(t-\tau, x-y)f(\tau, y) \, dy \, d\tau. 
\end{align*}
Using Remark~\ref{nonnegative} to deduce that $Z$ is a probability density gives
\begin{equation}\label{diff_quotient}\notag
\begin{split}
&\frac{1}{h}[v(t+h,x)-v(t,x)] \\
&= \dashint_{t}^{t+h}\int_{\Rn}Z(t+h-s, x-y)f(s,y) \, ds \, dy \\
&\quad+ \int_0^t \int_{\Rn}\frac{1}{h}[Z(t+h-s, x-y)-Z(t-s, x-y)]f(s,y) \, dy \, 
ds \\
&= \dashint_{t}^{t+h}\int_{\Rn}Z(t+h-s, x-y)[f(s,y)-f(s,x)] \, dy \, ds 
+\dashint_{t}^{t+h}f(s,x) \, ds \\
&\quad+ \int_0^t \int_{\Rn}\frac{1}{h}[Z(t+h-s, x-y)-Z(t-s, x-y)]f(s,y) \, dy \, 
ds. \\
\end{split}
\end{equation}
At this point we use the conditions~\eqref{f_hat_assumption} 
and~\eqref{g_assumption} to give some regularity for the right hand side \(f\). 
Our conditions guarantee that \(f(t,\cdot)\) is a H\"older continuous function 
with the H\"older exponent less than \(\min\{1,\beta\}\) uniformly in \(t\), 
i.e. we have an estimate
\[
|f(t,y)-f(t,x)| \le C|x-y|^\gamma,\quad t \ge 0,
\]
for any \(0<\gamma<\min\{1,\beta\}\). Using this we obtain
\begin{align*}
&\dashint_{t}^{t+h}\int_{\Rn}Z(t+h-s, x-y)[f(s,y)-f(s,x)] \, dy \, ds \\
&=\dashint_{0}^{h}\int_{\Rn}Z(s, x-y)|f(t+h-s,y)-f(t+h-s,x)| \, dy \, ds\\
&\lesssim \dashint_{0}^{s}\int_{\{|x-y| \ge s^{\alpha/\beta}\}}Z(s, x-y) \cdot 
|x-y|^\gamma \, dy \, ds\\
&\quad+ \dashint_{0}^{h}\int_{\{|x-y| < s^{\alpha/\beta}\}}Z(s, x-y)\cdot 
|x-y|^\gamma \, dy \, ds. 
\end{align*}
and continue by using Lemma~\ref{fund_sol_asympt} to 
conclude that
\begin{align*}
 \dashint_{0}^{h}\int_{\{|x-y| \ge s^{\alpha/\beta}\}}Z(s,x-y) \cdot |x-y|^\gamma 
\, dy \, ds &\lesssim \dashint_{0}^{h}s^\alpha \int_{h^{\alpha/\beta}}^\infty 
r^{-\beta-1+\gamma} \, dr  \, ds \\
 &\lesssim \dashint_{0}^{h} s^{\gamma\alpha /\beta} \, ds \lesssim 
h^{\gamma\alpha /\beta} \to 0
\end{align*}
as $h \to 0$. Utilizing Lemma~\ref{fund_sol_asympt} similarly in the second 
integral gives
\begin{align*}
 \dashint_{0}^{h} \int_{\{|x-y| < s^{\alpha/\beta}\}}Z(s,x-y) \cdot |x-y|^\gamma 
\, dy \, ds \lesssim h^{\gamma\alpha /\beta} \to 0
\end{align*}
as $h \to 0$. Here one needs to check different cases depending on the values of 
$\alpha$, $\beta$ and $d$. 

Altogether we have that
\[
\lim_{h \to 0}\frac{v(t+h,x)-v(t,x)}{h} = f(t,x) + \int_0^t 
\int_{\Rn}\frac{\partial Z(t-s, x-y)}{\partial t}f(s,y) \, dy \, ds
\]
and, therefore, the function $J^{1-\alpha}\Psi$ is continuously differentiable 
with respect to $t$.
\\

\noindent{\em Step III:}  We need to prove that the function \(\Psi\) 
satisfies the integro-partial differential equation. 
Our assumptions on \(f\) and the asymptotic behavior of \(Y\) 
guarantee that 
\(Y(t-\cdot,x-\cdot)f(\cdot,\cdot)\in L^1((0,t)\times\Rn)\). Therefore
\[
\Psi_2(t,x) \to 0 , \quad \text{as}\quad t 
\to 0,
\]
which means that
\[
\partial_t^\alpha\Psi_2(t,x)=\partial_t^{\alpha}(\Psi_2(t,x)-\Psi_2(0,x)).
\]

Notice that as a by-product of {\em Step II} we obtained 
\begin{equation}\label{sideproduct}
\partial_t^\alpha\Psi_2(t,x)
=f(t,x) + \int_0^t \int_{\Rn}\frac{\partial Z(t-s, x-y)}{\partial t}f(s,y) \, 
dy \, ds.
\end{equation}

We will show that \( 
(\partial_t^\alpha+(-\Delta)^{\beta/2})\Psi_2(t,x)=f(t,x)\). 
By~\eqref{sideproduct} it is enough to show that 
\[
(-\Delta)^{\beta/2}\Psi_2(t,x)=-\int_0^t \int_{\Rn}\frac{\partial Z(t-s, 
x-y)}{\partial t}f(s,y) \, 
dy \, ds.
\]

We start by 
calculating 
\[
(-\Delta)^{\beta/2}Y(t,x)=\mathcal{F}_{\xi\to 
x}^{-1}(|\xi|^\beta\widehat{Y}(t,\xi))(t,x).
\]
Recall from~\eqref{Y_fourier_trans} that
\[
\widehat{Y}(t, \xi)=(2\pi)^{-d/2} t^{\alpha-1}E_{\alpha, \alpha}(-|\xi|^\beta 
t^\alpha)=(2\pi)^{-d/2} 
t^{\alpha-1}H_{12}^{11}\bigl[ \abs{\xi}^\beta t^{\alpha}
\big| \begin{smallmatrix} (0,1) & \\ (0,1),& (1-\alpha,\alpha ) 
\end{smallmatrix} 
\big].
\]
Notice that in the following calculations we have to interpret the integral 
properly due to poor decay of 
\(\widehat{Y}(t,\cdot)\) at infinity, see Remark~\ref{rem_Htrans_formula}. 

Notice that $\widehat Y$ is a radial function of $\xi$ and
for radial functions we have in general that~\cite[Appendix B.5]{Graf04}
\[
\mathcal{F}(f)(\abs{\xi})=\abs{\xi}^{\frac{2-d}{2}}\int_{0}^{\infty
} f(r)r^ { d/2}J_{\frac{d-2}{2}}(r\abs{\xi}) \, dr,
\]
where $J_{(d-2)/2}$ is the modified Bessel function. For the 
definition, see~\cite{Wats95}. We use this formula together with 
Lemma~\ref{Fox_properties} (v) and (vi) to calculate
\begin{align*}
&\mathcal{F}^{-1}(|\xi|^\beta\widehat{Y}(t, \xi))(x)\\
&=(2\pi)^{-d/2}\abs{x}^{\frac{2-d}{2}}t^{\alpha-1}\int_{0}^{\infty}r^{d/2+\beta}
J_ { \frac { d-2 }{2}}(r\abs{x})H_{12}^{11}\bigl[ r^\beta t^{\alpha}
\big| \begin{smallmatrix} (0,1) & \\ (0,1),& (1-\alpha,\alpha ) 
\end{smallmatrix} 
\big] \, dr \\
&=\frac{2^\beta}{\pi^{n/2}}|x|^{-d-\beta}t^{\alpha-1}H_{32}^{12}\bigl[ 2^\beta 
|x|^{-\beta}t^{\alpha}
\big| \begin{smallmatrix} \l(1-\frac{d}{2}-\frac{\beta}{2}, \frac{\beta}{2}\r), 
&(0,1), &\l(-\frac{\beta}{2}, \frac{\beta}{2}\r) \\ (0,1),& (1-\alpha,\alpha ) 
\end{smallmatrix} 
\big] \\
&=\pi^{-d/2}|x|^{-d}t^{-1}H_{32}^{12}\bigl[ 2^\beta |x|^{-\beta}t^{\alpha}
\big| \begin{smallmatrix} \l(1-\frac{d}{2}, \frac{\beta}{2}\r), &(1,1), &\l(0, 
\frac{\beta}{2}\r) \\ (1,1),& (1,\alpha ) \end{smallmatrix}\big].
\end{align*}

On the other hand, combining the chain rule with Lemma~\ref{Fox_properties} (i), 
gives
\begin{equation}\label{Z_derivative}
\partial_t Z(t,x)=\alpha \pi^{-d/2}|x|^{-d}t^{-1}H_{43}^{13} \big[ 2^\beta 
t^\alpha |x|^{-\beta}\big| \begin{smallmatrix} (0,1), &\l(1-\frac{d}{2}, 
\frac{\beta}{2}\r), &(0,1), &\l(0, \frac{\beta}{2}\r) \\ (0,1),& (0,\alpha ), 
&(1,1) \end{smallmatrix} 
\big].
\end{equation}
Now by studying the Mellin transform $\mathcal H_{43}^{13}$ and using the 
properties of the Gamma function gives
\begin{align*}
&H_{43}^{13} \big[ 2^\beta t^\alpha |x|^{-\beta}\big| \begin{smallmatrix} (0,1), 
&\l(1-\frac{d}{2}, \frac{\beta}{2}\r), &(0,1), &\l(0, \frac{\beta}{2}\r) \\ 
(0,1),& (0,\alpha ), &(1,1) \end{smallmatrix} 
\big] \\
&= -\alpha^{-1}H_{32}^{12} \big[ 2^\beta t^\alpha |x|^{-\beta}\big| 
\begin{smallmatrix} \l(1-\frac{d}{2}, \frac{\beta}{2}\r), &(1,1), &\l(0, 
\frac{\beta}{2}\r) \\  (1,1 ), &(1,\alpha) \end{smallmatrix} 
\big].
\end{align*}
Inserting this into~\eqref{Z_derivative} yields
\[
\partial_t Z(t,x)=-\mathcal{F}^{-1}(|\xi|^\beta\widehat{Y}(t, \xi))(x).
\]
By the above calculation 
\begin{align*}
 &\mathcal F\l(\int_0^t \int_{\Rn}\frac{\partial Z(t-s, x-y)}{\partial t}f(s,y) 
\, dy \, ds\r)\\
 &=\int_0^t \mathcal{F}(\partial_t Z(t-s, \xi))\widehat{f}(s, \xi) \, ds 
=-|\xi|^\beta\int_0^t \widehat{Y}(t-s, \xi)\widehat{f}(s, \xi) \, ds. 
\end{align*}
Now using the growth condition of function $f$ (cf.~\eqref{f_hat_assumption}), 
we have that $|\cdot|^\beta 
\int_0^t\widehat{Y}(t-s,\cdot)\widehat{f}(s,\cdot)\,ds \in L^1(\Rn)$ and 
therefore it has a unique inverse Fourier transform. We obtain
\begin{align*}
&(-\Delta)^{\beta/2}\l(\int_0^t\int_{\Rn} Y(t-s,x-y)f(s,y)\, dy \, ds\r) \\
&= \mathcal F^{-1}\l(|\xi|^\beta\int_0^t \widehat{Y}(t-s, \xi)\widehat{f}(s, 
\xi) \, ds\r) \\
&=-\int_0^t \int_{\Rn}\frac{\partial Z(t-s, x-y)}{\partial t}f(s,y) \, dy \, ds.
\end{align*}
Therefore
\[
\big(\partial_t^\alpha+(-\Delta)^{\beta/2}\big)\Psi_2(t,x) = f(t,x),
\]
as claimed. 

Let us now study the first integral. By using the asymptotics of $Z$ as in {\em 
Step II}, it is straightforward to 
show that
\[
\int_{\Rn} Z(t, x-y) u_0(y) \, dy \to u_0 (x), \quad \text{as}\quad t \to 0.
\]
A similar argument as for $\Psi_2$ produces 
\[
\partial_t^\alpha\big[\int_{\Rn} Z(t,x-y)u_0(y)\, 
dy-u_0(x)\big]+(-\Delta)^{\beta/2}\int_{\Rn} Z(t,x-y)u_0(y)\, dy = 0.
\]
We omit the details.


Now \(\Psi\) 
satisfies the initial condition by the superposition principle.

\noindent{\em Step IV:} Finally we have to prove that \(\Psi\) is a jointly 
continuous function in \([0,\infty)\times\Rn\). The continuity at \(t=0\) 
is established in {\em Step III}. If \(t>0\), the continuity in both variables 
follows from our conditions given for \(u_0\) and \(f\), which guarantee that 
\(u_0\) and \(f\) are continuous and uniformly bounded. Then the asymptotics of 
\(Z\) and \(Y\) given in Lemmas~\ref{fund_sol_asympt} 
and~\ref{fund_sol_asympt_Y} together with the Lebesgue dominated convergence 
theorem imply the continuity. This finishes the 
proof.
\end{proof}

\section{Large-time behavior of mild solutions}

We begin by calculating an \(L^p\)-decay estimate for the fundamental solution 
\(Z\), which is given in the following lemma.

\begin{lemma}\label{Z_Lp_decay}
Let \(d\ge 1\), $0<\alpha \le 1$ and $0 < \beta \le2$. 
Then \(Z(t,\cdot)\in L^p(\Rn)\) for any $t>0$ and
\begin{equation}\label{fund_sol_decay1}
\| Z(t,\cdot)\|_{L^p(\Rn)}\lesssim t^{-\frac{\alpha 
d}{\beta}(1-\frac{1}{p})},\quad t>0,
\end{equation}
for every \(1\le p<\kappa_3(\beta, d)\), where
\begin{equation}\label{kappa_beta}
\kappa_3=\kappa_3(\beta, d):=\begin{cases} 
\frac{d}{d-\beta}, \quad \text{if} \ \ d > \beta,\\
\infty, \quad \text{otherwise.}\end{cases}
\end{equation}
Moreover, if $\alpha=1$ or $1=d \le \beta$, then~\eqref{fund_sol_decay1} holds for all $p \in [1, \infty]$. Finally, for 
\(d>\beta\) and \(0<\alpha<1\), we obtain
\[
\|Z(t,\cdot)\|_{L^{\frac d{d-\beta}, \infty}} \lesssim t^{-\alpha}, \quad t>0. 
\]

\begin{proof}
We begin by decomposing the $L^p$-integral of $Z$ as
\[
\|Z(t,\cdot)\|_{L^p}^p \le \int_{\{R\ge 1\}}Z(t,x)^p\,dx+\int_{\{R\le 1\}}Z(t,x)^p\,dx.
\]
In view of Lemma~\ref{fund_sol_asympt}, we have for all dimensions $d$ and for 
all values $1\le p<\infty$ that
\begin{align*}
\int_{\{R\ge 1\}}Z(t,x)^p\,dx  & \lesssim \int_{\{R\ge 1\}} t^{\alpha p}|x|^{-dp-\beta p}\,dx \\
& \lesssim \int_{t^\frac{\alpha}{\beta}}^\infty  t^{\alpha p}r^{-dp-\beta p}r^{d-1}\,dr \lesssim t^{ -\frac{\alpha d }{\beta}(p-1)},
\end{align*}
and thus
\begin{equation} \label{goodpart_Z}
\left(\int_{\{R\ge 1\}}Z(t,x)^p\,dx\right)^\frac{1}{p} \lesssim t^{ 
-\frac{\alpha d }{\beta}\,\left(1-\frac{1}{p}\right)}\quad \textrm{for 
all $1<p<\infty$ and $t>0$}.
\end{equation}

We come now to the estimate for the integral where $R\le 1$. In the case 
\(\alpha=1\) or $\beta > d$ and $ 0<\alpha<1$, we have for all $1\le p<\infty$ 
that
\begin{align*}
\int_{\{R\le 1\}}Z(t,x)^p\,dx & \lesssim \int_{\{R\le 1\}} t^{-\frac{\alpha 
dp}{\beta}}\,dx  \lesssim
\int_0^{t^\frac{\alpha }{\beta}} t^{-\frac{\alpha d p}{\beta}}r^{d-1}\,dr 
\lesssim t^{-\frac{\alpha dp}{\beta}+\frac{\alpha d}{\beta}}.
\end{align*}
If $\beta =d$ and $0<\alpha <1$ we estimate
\begin{align*}
\int_{\{R\le 1\}}Z(t,x)^p\,dx & \lesssim \int_{\{R\le 1\}} t^{-\alpha p}(|\log(|x|^\beta t^{-\alpha})|+1)^p\,dx\\
& \lesssim \int_0^{t^\frac{\alpha}{\beta}} t^{-\alpha p} \left(|\log (r^\beta t^{-\alpha})|+1\right)^p\,r^{d-1}\,dr\\
& \lesssim \int_0^1 t^{-\alpha p+\alpha d/\beta} \left(|\log (s^\beta)|+1\right)^ps^{d-1}\,ds \\
&\lesssim   t^{-\alpha p+\alpha}=t^{-\frac{\alpha d}{\beta}(p-1)},
\end{align*}
for all $1\le p<\infty$. Note that the condition \(\beta>d\) can only happen if \(d=1\).

Finally, if $0<\beta <d$ and $0<\alpha <1$ we have
\begin{align*}
\int_{\{R\le 1\}}Z(t,x)^p\,dx & \lesssim \int_{\{R\le 1\}} t^{-\alpha p}|x|^{-dp+\beta p}\,dx
\lesssim \int_0^{t^\frac{\alpha}{\beta}} t^{-\alpha p} r^{(-d+\beta)p} r^{d-1}\,dr \\
& \lesssim   t^{-\alpha p+\frac{\alpha}{\beta}[d-(d-\beta)p]} 
\lesssim t^{-\frac{\alpha d}{2}\, (p-1)},
\end{align*}
whenever the last integral is finite, that is, whenever
\[
p<\,\frac{d}{d-\beta}=\kappa_3(\beta, d).
\]

Combining the previous estimates we see that
\begin{equation} \label{badpart}
\left(\int_{\{R\le 1\}}Z(t,x)^p\,dx\right)^\frac{1}{p} \lesssim t^{ 
-\frac{\alpha d }{\beta}\,\left(1-\frac{1}{p}\right)}\quad \textrm{for 
all $1\le p<\kappa_3(\beta, d)$ and $t>0$}.
\end{equation}

Observe that by Lemma~\ref{fund_sol_asympt} we have $Z(t,\cdot)\in L^\infty(\R)$ for all $t>0$, provided $\alpha=1$ or 
\(\beta<d\), and moreover, we have the estimate
\[
\|Z(t,x)\|_{L^\infty} \lesssim t^{-\frac{\alpha d}{\beta}},
\]
which proves the second statement.

For the weak-$L^p$-estimate we set $p=\frac{d}{d-\beta}$. We need to estimate
\[
\|Z(t, \cdot)\|_{L^p,\,\infty}=\sup \left\{\lambda\, d_{Z(t,x)}(\lambda)^{\frac{1}{p}}:\,\lambda>0\right\},
\]
where
\[
d_{Z(t,x)}(\lambda)=|\{x\in \R^d:\,Z(t,x)>\lambda\}|
\]
denotes the distribution function of $Z(t,x)$. Using again the similarity variable $R=t^{-\alpha}|x|^\beta$ we have
\begin{equation}\label{weak_triangle}
\begin{split}
&\|Z(t, \cdot)\|_{L^{p, \infty}}  \\
&\le 2\left(\|Z(t,x)\chi_{\{R\le 1\}}(t)\|_{L^{p, \infty}}+
\|Z(t,x)\chi_{\{R\ge 1\}}(t)\|_{L^{p, \infty}}\right).
\end{split}
\end{equation}
Employing (\ref{goodpart_Z}), we find that
\[
\|Z(t,x)\chi_{\{R\ge 1\}}(t)\|_{L^{p, \infty}}
\le \|Z(t,x)\chi_{\{R\ge 1\}}(t)\|_{L^p}
\le C t^{ -\frac{\alpha d }{\beta}\,\left(1-\frac{1}{p}\right)}
= C t^{-\alpha}.
\]
For the term with $R\le 1$ we use the case $0<\beta <d$ of Lemma~\ref{fund_sol_asympt} to estimate
\begin{align*}
d_{Z(t,x)\chi_{\{R\le 1\}}(t)}(\lambda) & =|\{x\in \R^d:\,Z(t,x)>\lambda\;\mbox{and}\;R\le 1\}|\\
& \le |\{x\in \R^d:\,\lambda< C t^{-\alpha} |x|^{-d+\beta}\}|\\
& = |\{x\in \R^d:\,|x|< \left(C t^{-\alpha}\lambda^{-1}\right)^{\frac{1}{d-\beta}}\}|\\
& \le C_1 \left( t^{-\alpha}\lambda^{-1}\right)^{\frac{d}{d-\beta}}.
\end{align*}
This shows that
\[ d_{Z(t,x)\chi_{\{R\le 1\}}(t)}(\lambda)^{1/p}\le C_1^{1/p} t^{-\alpha}\lambda^{-1},
\]
and thus
\[
\|Z(t,x)\chi_{\{R\le 1\}}(t)\|_{L^{p,\infty}}\lesssim\,t^{-\alpha}.
\]
This finishes the proof.
\end{proof}
\end{lemma}

As a simple consequence of the above lemma we obtain the following decay result. 

\begin{proposition}\label{prop_homog}
Let \(d\ge 1\), \(0<\alpha\le 1\) and \(0<\beta\le 2\). Assume that \(u\) 
is the mild solution of equation~\eqref{classical_equation} with \(f\equiv0\) and \(u_0\in 
L^q(\Rn)\), where $1 \le q \le \infty$. 
Then the following hold:
\begin{itemize}
\item[(i)] if $q=\infty$ we have
\[
\|u(t,\cdot)\|_{L^\infty(\Rn)}\lesssim \|u_0\|_{L^\infty(\Rn)},\quad t>0;
\]
\item[(ii)] if $1 \le q < \infty$ and $d > q\beta$, we have for every $r \in [q, \frac{qd}{d-q\beta})$ that
\begin{equation}\label{homog_sol_decay_est}
\|u(t,\cdot)\|_{L^r(\Rn)}\lesssim t^{-\frac{\alpha d}{\beta}(\frac{1}{q}-\frac{1}{r})},\quad t>0,
\end{equation}
and if, in addition, $0<\alpha <1 < d$ we obtain
\[
\|u(t,\cdot)\|_{L^{\frac{qd}{d-q\beta}, \infty}(\Rn)}\lesssim t^{-\alpha},\quad t>0;
\]
\item[(iii)] if $1 \le q < \infty$ and $d = q\beta$, the estimate~\eqref{homog_sol_decay_est} holds for every $r \in [q, \infty)$;

\item[(iv)] if $d < q\beta$ or $\alpha=1$, the estimate~\eqref{homog_sol_decay_est} holds for every $r \in [q, \infty]$.
\end{itemize}
\end{proposition}

\begin{proof}
Let $p$ be defined via
\begin{equation}\label{pqr_cond}
1+\frac{1}{r}=\frac{1}{p}+\frac{1}{q}.
\end{equation}
For such $p,q$ and $r$ we may use Young's inequality for convolutions to obtain
\begin{equation}\label{Z_Young}
\|\Psi(t,\cdot)\|_{L^r}=\|Z(t, \cdot) \star u_0(\cdot)\|_{L^r} \le \|Z(t, \cdot)\|_{L^p}\|u_0\|_{L^q}.
\end{equation}
The idea is now to use Lemma~\ref{Z_Lp_decay} to estimate the $L^p$-decay of $Z$ on the right hand side of the above estimate. We only need to consider different cases corresponding to the different choices of the parameters.

Recall that by Lemma~~\ref{Z_Lp_decay} we obtain 
\begin{equation}\label{Z_decay}
\|Z(t, \cdot)\|_{L^p(\Rn)} \lesssim t^{-\frac{\alpha d}{\beta}(1-\frac 1p)},
\end{equation}
for $1 \le p < \kappa_3(\beta, d)$, where $\kappa_3$ is as in~\eqref{kappa_beta}. Now, claim (i) follows directly from choosing $p=1$, $r=\infty$ and $q=\infty$ in~\eqref{Z_Young}. 

On the other hand, a straightforward calculation shows that for $r \in [q, \frac{qd}{d-q\beta})$, we have $1 \le p < \frac{d}{d-\beta}$. We again use the above $L^p$-estimate for $Z$ together with~\eqref{Z_Young} to obtain the claim. 

For $d>q\beta$ and $r=\frac{qd}{d-q\beta}$, we have from~\eqref{pqr_cond} that $p=\frac{d}{d-\beta}$. Now we may use the second part of Lemma~\ref{Z_Lp_decay} to obtain that if $d \ge 2$ and $0 < \alpha < 1$, then
\[
\|Z(t, \cdot) \|_{L^{\frac{d}{d-\beta}, \infty}(\Rn)} \lesssim t^{-\alpha}, \quad t>0,
\]
which together with Young's inequality for weak $L^p$-spaces gives
\begin{align*}
\|u(t,\cdot)\|_{L^{\frac{qd}{d-q\beta},\infty}(\Rn)}&\lesssim\|Z(t,\cdot)\|_{L^{\frac d{d-\beta},\infty}(\Rn)}
\|u_0\|_{L^q(\Rn)} \\
& \lesssim t^{-\alpha},
\end{align*}
as required.

For (iii), observe that inserting $q=d/\beta$ and $p \in [1, \frac d{d-\beta})$ in~\eqref{pqr_cond} gives $r \in [q, \infty)$. Similarly, inserting $q  \in (d/\beta, \infty]$ and $p \in [1, \frac d{d-\beta})$ in~\eqref{pqr_cond} gives $r \in [q, \infty]$. This yields the first claim of (iv). If $\alpha=1$ we, in turn, by Lemma~\ref{Z_Lp_decay} obtain the $L^p$-decay~\eqref{Z_decay} for any $1 \le p \le \infty$ and we may again use Young's inequality, as in~\eqref{Z_Young}, to obtain the claim for any $r \in [q, \infty]$.
\end{proof}



We continue by studying the above type of results for the inhomogeneous equation. First we need the \(L^p\)-decay estimates for the fundamental solution $Y$. 

\begin{lemma}\label{Y_decay_estimate}
Let \(d\ge 1\), $0 <\alpha \le1$ and $0<\beta \le 2$. Then \(Y(t,\cdot)\in L^p(\Rn)\) and 
\begin{equation}\label{fund_sol_decay2}
\| Y(t,\cdot)\|_{L^p(\Rn)}\lesssim t^{\alpha-1-\frac{\alpha 
d}{\beta}(1-\frac{1}{p})},\quad t>0,
\end{equation}
 for every \(1\le p<\kappa_2\), where
\[
\kappa_2=\kappa_2(\beta, d)=\begin{cases} \frac{d}{d-2\beta}, \quad\text{if} \ \ d > 2\beta, \\
\infty, \quad\text{otherwise}. \end{cases}
\] 
At the borderline $p=\kappa_2$, we also have  for $d > 2\beta$ that $Y(t,\cdot)$ belongs to $L^{\frac{d}{d-2\beta}, \infty}(\Rn)$ and
\[
\|Y(t,\cdot)\|_{L^{\frac{d}{d-2\beta}, \infty}} \lesssim t^{-1-\alpha}, \quad t>0. 
\]
Finally, if $\alpha=1$ or $d < 2\beta$, estimate~\eqref{fund_sol_decay2} holds for all $p \in [1,\infty]$.

\begin{proof}
The proof is similar to that of the function $Z$. We give the proof in the case 
$d < 2 \beta$ and \(0<\alpha<1\) as an example. We begin by decomposing the 
$L^p$-integral of $Y$ as
\[
\|Y(t,\cdot)\|_{L^p}^p =\int_{\{R\ge 1\}}Y(t,x)^p\,dx+\int_{\{R\le 
1\}}Y(t,x)^p\,dx.
\]
By Lemma~\ref{fund_sol_asympt_Y}, we have for all dimensions $d$ and for all 
values $1\le p<\infty$ that
\begin{align*}
\int_{\{R\ge 1\}}Y(t,x)^p\,dx  & \lesssim \int_{\{R\ge 1\}} t^{2\alpha 
p-p}|x|^{-dp-\beta p}\,dx \\
& \lesssim t^{2\alpha p -p} \int_{t^\frac{\alpha}{\beta}}^\infty  r^{-dp-\beta 
p}r^{d-1}\,dr \lesssim t^{ (\alpha-1)p -\frac{\alpha d}{\beta}(p-1)},
\end{align*}
and thus
\begin{equation} \label{goodpart}
\left(\int_{\{R\ge 1\}}Y(t,x)^p\,dx\right)^\frac{1}{p} \lesssim t^{\alpha-1 
-\frac{\alpha d }{\beta}\,\left(1-\frac{1}{p}\right)}\quad\textrm{for 
all $1\le p<\infty$ and $t>0$}.
\end{equation}

We come now to the estimate for the integral where $R\le 1$. Again, by 
Lemma~\ref{fund_sol_asympt_Y}, we have 
\begin{align*}
\int_{\{R\le 1\}}Y(t,x)^p\,dx & \lesssim \int_{\{R\le 1\}} 
t^{(\alpha-1)p-\frac{\alpha d}{\beta}p}\,dx  \lesssim
t^{(\alpha-1)p-\frac{\alpha d}{\beta}p}\int_0^{t^\frac{\alpha}{\beta}} 
r^{d-1}\,dr \\
&\lesssim t^{ (\alpha-1)p -\frac{\alpha d}{\beta}(p-1)}
\end{align*}
for all $1\le p<\infty$, which finishes the proof of the first statement in 
this case. Since in this case even \(Y(t,\cdot)\in L^\infty(\Rn)\), we see that 
the second statement holds as well. 

The weak-$L^p$ estimate is done similarly to Lemma~\ref{Z_Lp_decay}. We omit the 
details.
\end{proof}

\end{lemma}

Again, we may use the above estimates to prove a decay result concerning the source term $f$. Here we need to impose a decay condition similar to~\eqref{decay_cond_f} for the source term. We obtain the following proposition. 

\begin{proposition}\label{proposition_source_term}

Let $d \ge 1$, $0<\alpha \le 1$ and \(0<\beta\le 2\). Assume that $u$ is the mild 
solution of equation~\eqref{classical_equation} with $u_0=0$ and $f(t, \cdot) \in 
L^q(\Rn)$ for each $t \ge 0$ and for some $q \in [1, \infty)$. Assume further that \(f\) satisfies the 
decay condition
\begin{equation}\label{decayf}
\|f(t,\cdot)\|_{L^q(\R^d)}\lesssim (1+t)^{-\gamma},\quad t> 0,
\end{equation}
for some $\gamma >0$. Then we have in case $\gamma\neq 1$
\begin{itemize}
\item[(i)] if $1 \le q <\infty$ and $d > q\beta$, we have for every $r \in [q, \frac{qd}{d-q\beta})$ that
\begin{equation}\label{sol_decay1_gen_beta}
\|u(t,\cdot)\|_{L^r(\Rn)}\lesssim t^{\alpha-\min\{1,\gamma\}-\frac{\alpha 
d}{\beta}(\frac{1}{q}-\frac{1}{r})},\quad t>0;
\end{equation}
\item[(ii)] if $1 < q <\infty$ and $ d \le q\beta$, the estimate~\eqref{sol_decay1_gen_beta} holds for every $r \in [q, \infty)$.
\end{itemize}
In the case $\gamma=1$, the assertions (i) and (ii) are valid with (\ref{sol_decay1_gen_beta}) replaced by
\begin{equation}\label{sol_decay1_gen_beta2}
\|u(t,\cdot)\|_{L^r(\Rn)}\lesssim t^{\alpha-1-\frac{\alpha 
d}{\beta}(\frac{1}{q}-\frac{1}{r})}\log(1+t),\quad t>0.
\end{equation}

\end{proposition}

\begin{proof}
The proof is now an easy application of the integral form of the Minkowsky 
inequality, the Young inequality for convolutions and 
Lemma~\ref{Y_decay_estimate}.

Using the Minkowsky inequality, we have
\[
\|u(t,\cdot)\|_{L^r(\Rn)}\le 
\int_0^t\Big(\int_{\Rn}\Big|\int_{\Rn}Y(t-s,x-y)f(s,y)dy\Big|^r dx\Big)^{1/r}ds
\]
for \(1\le r<\infty\). 

Similarly as in the proof of Proposition~\eqref{prop_homog}, we choose $p$ such that
\begin{equation}\label{pqr_cond2}
1+\frac1r=\frac1p+\frac1q.
\end{equation}
Then the Young inequality for convolution yields 
\[
\|u(t,\cdot)\|_{L^r(\Rn)}\leq \int_0^t 
\|Y(t-s,\cdot)\|_{L^p(\Rn)}\|f(s,\cdot)\|_{L^q(\Rn)} ds.
\]
We split the integral into two parts as follows
\begin{equation}\label{splitting}
\begin{split}
&\int_0^t\|Y(t-s,\cdot)\|_{L^p(\Rn)}\|f(s,\cdot)\|_{L^q(\Rn)}ds\\
=&\Big(\int_0^{
t/2 }+\int_{t/2}^t\Big)\|Y(t-s,\cdot)\|_{L^p(\Rn)}\|f(s,\cdot)\|_{L^q(\Rn)}ds=: 
I_1+I_2.
\end{split}
\end{equation}
Recall that by Lemma~\ref{Y_decay_estimate} we have
\begin{equation}\label{Y}
\| Y(t,\cdot)\|_{L^p(\Rn)}\lesssim t^{\alpha-1-\frac{\alpha 
d}{\beta}(1-\frac{1}{p})},\quad t>0,
\end{equation}
for $1 \le p < \kappa_2$ where
\begin{align*}
\kappa_2=\begin{cases} \frac{d}{d-2\beta}, \quad\text{if} \ \ d > 2\beta, \\
\infty, \quad\text{otherwise}. \end{cases}
\end{align*}
If $d >2 q \beta$, for $r \in [q, \frac{qd}{d-2 q\beta}) \supset [q, \frac{qd}{d-q\beta})$ we  obtain from~\eqref{pqr_cond2} that correspondingly $p \in [1, \frac{d}{d-2\beta})$. Therefore, we may use~\eqref{Y} to estimate the $L^p$-norm of $Y$ in~\eqref{splitting}. On the other hand, if $d \le 2q\beta$, the different values of $p \in [1, \kappa_2)$ yield the corresponding choices of $r$ in $[q, \infty)$ and, thus, we may again use~\eqref{Y} to estimate~\eqref{splitting}. 

We continue the estimate by using~\eqref{Y}. For the first integral we observe that \(\frac{t}{2}\le t-s\le t\) and hence 
~\eqref{Y} together with the decay 
condition~(\ref{decayf}) implies
\[
I_1\lesssim t^{\alpha-1-\frac{\alpha 
d}{\beta}(1-\frac{1}{p})}\int_0^{t/2}\|f(s,\cdot)\|_{L^q(\Rn)}ds\lesssim 
t^{\alpha-1-\frac{\alpha 
d}{\beta}(1-\frac{1}{p})}\int_0^t (1+s)^{-\gamma}ds,
\]
which gives the desired estimate for \(I_1\).

In the second integral we need to take care of the singularity of 
\(Y(t,\cdot)\) at \(t=0\). The integral converges if and only if
\[
\alpha-1-\frac{\alpha d}{\beta}(1-\frac{1}{p})>-1.
\]
This gives $1 \le p <\frac{d}{d-\beta}$, provided $d > \beta$. If $d=\beta$, this holds for all $p \in [1, \infty)$, and if $d < \beta$, this estimate is always true. Observe that this restriction gives the different choices of $r$ in the items (i) and (ii) of the claim. We obtain
\[
\begin{split}
I_2&\lesssim \int_{t/2}^t(1+s)^{-\gamma}\|Y(t-s,\cdot)\|_{L^p(\Rn)}ds\lesssim 
t^{-\gamma}\int_0^{t/2}s^{\alpha-1-\frac{\alpha 
d}{\beta}(1-\frac{1}{p})}ds\\
&\lesssim t^{\alpha-\gamma-\frac{\alpha 
d}{\beta}(1-\frac{1}{p})},
\end{split}
\]
for all $\gamma>0$. So $I_2$ decays faster than $I_1$ if $\gamma\ge 1$, whereas for $\gamma\in (0,1)$ we obtain
the same decay rates. Observe also that, similarly as in Proposition~\ref{prop_homog}, the restriction $1 \le p < \frac{d}{d-\beta}$ plays a role only if $d \ge q\beta$. In this case, we obtain directly from~\eqref{pqr_cond2} that $r \in [1, \frac{qd}{d-q\beta})$. 
\end{proof}

The last step towards the proof of Theorem~\ref{zuazua} is the following gradient $L^p$-estimate for $Z$.

\begin{lemma}\label{G_Lp_gradient}
Let $d \in\mathbb{Z}_+$ and \(\kappa_{1}(\beta, d)\) be as in 
Section~\ref{section:preliminaries}. Then $\nabla Z(t,\cdot)$ belongs to 
$L^p(\Rn; \Rn)$ for all $t>0$ and $1\le p < \kappa_1(\beta, d)$, and 
there holds
\begin{equation}\label{gradient_Lp}
\|\nabla Z(t,\cdot)\|_{L^p(\Rn;\Rn)} \lesssim 
t^{-\frac{\alpha}{\beta}-\frac{\alpha d}{\beta}\l(1-\frac{1}{p}\r)}, \quad t>0.
\end{equation}
The estimate~\eqref{gradient_Lp} remains valid for $d=1$, $\beta=2$ and 
$p=\infty$. 

Moreover, if $p=\kappa_1(\beta, d)$, then we have that $\nabla 
Z(t,\cdot)$ belongs to $L^{p, \infty}(\Rn; \Rn)$ for all $t>0$ and
\[
\|\nabla Z(t,\cdot)\|_{L^{p, \infty}(\Rn;\Rn)} \lesssim t^{-\alpha}, \quad t>0. 
\]
\begin{proof}
The proof is very similar to that of Lemma~\ref{Z_Lp_decay}. Let $R=|x|^{\beta}t^{-\alpha}$ be the similarity variable. Let's first divide the object of our study into two parts:
\begin{align*}
&\int_{\Rn} |\nabla Z(t,x)|^p \, dx \\
&= \int_{\{R \le 1\}} |\nabla Z(t,x)|^p \, dx + \int_{\{R \ge 1\}} |\nabla Z(t,x)|^p \, dx =: I_1+I_2.
\end{align*}
For the first term, we may use Lemma~\ref{G_gradient} to get

\begin{equation}\label{badterm}
\begin{split}
I_1 =\int_{\{R \le 1\}} |\nabla Z(t,x)|^p \, dx &\lesssim \int_{\{|x| \le t^{\alpha/\beta}\}} |x|^{(-d+\beta+1)p}t^{-\alpha p} \, dx \\
& \lesssim t^{-\alpha p} \int_{0}^{t^{\alpha / \beta}} r^{(-d+\beta-1)p+d-1} \, 
dr\\
&\lesssim t^{-\frac{\alpha p}{\beta}-\frac{\alpha d}{\beta}(p-1)}.
\end{split}
\end{equation}
provided the last integral is finite, that is
\[
1\le p<\kappa_1(\beta, d).
\]

For the second term we again use Lemma~\ref{G_gradient} and obtain
\begin{equation}\label{goodterm}
\begin{split}
I_2 =\int_{\{R \ge 1\}} |\nabla Z(t,x)|^p \, dx &\lesssim  \int_{\{|x| \ge t^{\alpha/\beta}\}} |x|^{-(d+\beta+1)p}t^{\alpha p} \, dx \\
& \lesssim  t^{\alpha p} \int_{\{r  \ge t^{\alpha /\beta}\}}r^{-(d+\beta+1)p+d-1} \, dr \\
&\lesssim t^{-\frac{\alpha p}{\beta}-\frac{\alpha d}{\beta}(p-1)}
\end{split}
\end{equation}
for all \(1\le p<\infty\). Thus we obtain the first part of the lemma. 

If \(\beta\ge d+1\), which is equivalent with \(\beta=2\) and \(d=1\), we see 
from Lemma~\ref{G_gradient} that \(\nabla Z(t,\cdot)\) is indeed bounded. 
Therefore the second statement holds as well.

Let now 
$p=\kappa_1(\beta, d)$. 
Similarly as~\eqref{weak_triangle}, we obtain
\begin{align*}
&\|\nabla Z(t,x)(t,\cdot)\|_{L^{p,\,\infty}} \\
& \le 2\left(\|\nabla Z(t,x)(t, \cdot)\chi_{\{R\le 1\}}(t)\|_{L^{p,\,\infty}}+
\|\nabla Z(t,x)(t,\cdot)\chi_{\{R\ge 1\}}(t)\|_{L^{p,\,\infty}}\right).
\end{align*}
Employing estimate~\eqref{goodterm} gives
\begin{align*}
\|\nabla Z(t,x) \chi_{\{R\ge 1\}}(t)\|_{L^{p,\infty}}
&\le \|\nabla Z(t,x) \chi_{\{R\ge 1\}}(t)\|_{L^{p}} \\
&\lesssim t^{-\frac{\alpha}{\beta} -\frac{\alpha d }{\beta }\,\left(1-\frac{1}{p}\right)}
\lesssim t^{-\alpha}.
\end{align*}
For the term with $R\le 1$ we use Lemma~\ref{G_gradient} to estimate as follows.
\begin{align*}
d_{|\nabla Z(t,x)\chi_{\{R\le 1\}}(t)|}(\lambda) & =|\{x\in \Rn:\,|\nabla Z(t,x)|>\lambda\;\mbox{and}\;R\le 1\}|\\
& \le |\{x\in \Rn:\,\lambda< C t^{-\alpha} |x|^{-d+\beta-1}\}|\\
& = |\{x\in \Rn:\,|x|< \left(C t^{-\alpha}\lambda^{-1}\right)^{\frac{1}{d-\beta+1}}\}|\\
& \le C_1 \left( t^{-\alpha}\lambda^{-1}\right)^{\frac{d}{d-\beta+1}}.
\end{align*}
This shows that
\[ d_{|\nabla Z(t,x)\chi_{\{R\le 1\}}(t)|}(\lambda)^{1/p}\le C_1^{1/p} t^{-\alpha}\lambda^{-1},
\]
and thus
\[
\|\nabla Z(t,x)\chi_{\{R\le 1\}}(t)\|_{L^{p,\,\infty}}\lesssim\,t^{-\alpha},
\]
which finishes the proof.
\end{proof}
\end{lemma}

\subsection{Proof of Theorem~\ref{zuazua}}

Now we are ready to prove Theorem~\ref{zuazua}.

\begin{proof}[Proof of Theorem~\ref{zuazua}]

We split the proof into two parts. We first study the estimates for $Z$. The estimate for $Y$ is substantially more involved and we do it after studying $Z$.

\noindent {\em The estimates for $Z$:} The strategy of the proof here is the same as in \cite[p.\ 14, 15]{Zuaz03}. Suppose first that $u_0\in L^1(\Rn)$ is such that $\int_{\Rn} |x|\,|u_0(x)|\,dx<\infty$.
By Lemma \ref{decomp} there exists $\phi\in L^1(\Rn;\Rn)$ such that
\[
u_0=M_{init}\delta_0+\mbox{ div}\,\phi
\]
and $\|\phi\|_{L^1}\le C_d\||x|u_0\|_{L^1}$. Consequently,
\begin{align*}
u_{init}(t,x)&=M_{init}\left(Z(t,\cdot)\star \delta_0\right)(x)+\left(Z(t,\cdot)\star \mbox{ div}\,\phi(\cdot)\right)(x)\\
& =M_{init} Z(t,x)+(\nabla Z(t,\cdot)\star \phi)(x),
\end{align*}
which yields
\begin{equation} \label{asym1}
u_{init}(t,x)-M_{init} Z(t,x)=(\nabla Z(t,\cdot)\star \phi)(x).
\end{equation}
By Young's inequality it follows that for any $1\le p<\kappa_1(\beta, d)$
\begin{align*}
\|u_{init}(t, \cdot)-M_{init}Z(t, \cdot)\|_{L^p}\le \|\nabla Z(t, \cdot)\|_{L^p} \|\phi\|_1 &\lesssim  \|\nabla Z(t, \cdot)\|_{L^p} \||x|u_0\|_{L^1} \\
&\lesssim
t^{-\frac{\alpha}{\beta}-\frac{\alpha d}{\beta}\l(1-\frac{1}{p}\r)},
\end{align*}
where we used Lemma \ref{G_Lp_gradient}. Hence
\[
t^{\frac{\alpha d }{
\beta}\,\left(1-\frac{1}{p}\right)}\|u_{init}(t, \cdot)-M_{init}Z(t, \cdot)\|_{L^p}\lesssim t^{ -\frac{\alpha}{\beta}},
\]
which is the first part of assertion (ii). The second part follows from
(\ref{asym1}) by applying Young's inequality for weak $L^p$-spaces \cite[Theorem 
1.2.13]{Graf04}.

To prove (i) we choose a sequence $(\eta_j)\subset C_0^\infty(\Rn)$ such that
$\int_{\Rn} \eta_j\,dx=M_{init}$ for all $j$ and $\eta_j \rightarrow u_0$ in $L^1(\Rn)$. For each $j$ by Part (a) and by Lemma~\ref{Z_Lp_decay} we obtain
\begin{align*}
&\|u_{init}(t, \cdot)-M_{init}Z(t, \cdot)\|_{L^p} \\
& \le \|Z(t, \cdot)\star(u_0-\eta_j)\|_{L^p}+
  \|Z(t, \cdot)\star \eta_j-M_{init}Z(t, \cdot)\|_{L^p}\\
  & \le  \|Z(t, \cdot)\|_{L^p} \|u_0-\eta_j\|_{L^1}+C(j)\,t^{ -\frac{\alpha}{\beta} -\frac{\alpha d }{\beta}\,\left(1-\frac{1}{p}\right)}\\
  & \le C_1 t^{ -\frac{\alpha d }{\beta}\,\left(1-\frac{1}{p}\right)}\|u_0-\eta_j\|_{L^1}+
C(j)\,t^{ -\frac{\alpha}{\beta} -\frac{\alpha d }{\beta}\,\left(1-\frac{1}{p}\right)},
\end{align*}
and therefore
\[
t^{ \frac{\alpha d }{\beta}\,\left(1-\frac{1}{p}\right)}\|u_{init}(t, \cdot)-M_{init}Z(t, \cdot)\|_{L^p}\le C_1\|u_0-\eta_j\|_{L^1}+C(j)\,t^{ -\frac{\alpha}{\beta}},
\]
which implies
\[
\limsup_{t\to \infty}\,t^{ \frac{\alpha d }{\beta}\,\left(1-\frac{1}{p}\right)}\|u_{init}(t, \cdot)-M_{init}Z(t, \cdot)\|_{L^p}\le C_1\|u_0-\eta_j\|_{L^1}.
\]
Assertion (i) follows by sending $j\to \infty$. This finishes the decay estimates for $Z$. We continue with $Y$. \\

\noindent{\em The estimate for $Y$:} Next we turn to study $u_{forc}$. We split \(M_{forc}\) into two parts as follows
\[
M_{forc}=\int_{0}^t\int_{\R^d} f(\tau,y)\, d y\, d\tau+\int_{t}^\infty\int_{\R^d} 
f(\tau,y)\, d y\, d\tau 
\]
and note that
\[
\begin{split}
&t^{1+\frac{\alpha 
d}{\beta}(1-\frac{1}{p})-\alpha}\Big\|Y(t,\cdot)\int_{t}^\infty\int_{\R^d} 
f(\tau,y)\, d 
y\, d\tau\Big\|_{L^p} \\
\le& t^{1+\frac{\alpha 
d}{\beta}(1-\frac{1}{p})-\alpha}\|Y(t,\cdot)\|_{L^p}\int_t^\infty\int_{\R^d} 
|f(\tau,y)|\, d y\, d\tau\\
\le&\int_t^\infty\int_{\R^d} 
|f(\tau,y)|\, d y\, d\tau \to 0 
\end{split}
\]
as \(t\to\infty\). Here we used Lemma~\ref{Y_decay_estimate} to obtain
\[
\|Y(t,\cdot)\|_{L^p}\sim t^{\alpha-\frac{\alpha 
d}{\beta}(1-\frac{1}{p})-1},\quad t\to\infty.
\]

Therefore it suffices to prove that
\[
t^{1+\frac{\alpha 
d}{\beta}(1-\frac{1}{p})-\alpha}\Big\|\int_0^t 
(Y(t-\tau,\cdot)\star f(\tau,\cdot))\, d\tau-Y(t,\cdot)\int_0^t\int_{\R^d}f(\tau,
y)\, d y\, d\tau\Big\|_{L^p}\to 0,
\]
as \(t\to\infty\).

To prove the assertion, we fix \(0<\delta<\frac12\), and decompose the set of 
integration \((0,t)\times\R^d\) into two parts
\begin{align*}
\Omega_1(t)&=(0,\delta t)\times\{ y\in\R^d\,:\,|y|\le (\delta 
t)^{\alpha/\beta}\},\\
\Omega_2(t)&=(0,t)\times\R^d\setminus \Omega_1(t).
\end{align*}

Let us start with the set \(\Omega_1(t)\). We estimate by using the integral form of the Minkowsky inequality in the case \(1\le 
p<\infty\) to obtain
\begin{equation} \label{minkow}
\begin{split}
&\Big\|\iint_{\Omega_1(t)} 
[Y(t-\tau,\cdot-y)-Y(t,\cdot)]f(\tau,y)\, d y\, d\tau\Big\|_{L^p}\\
\le&\iint_{\Omega_1(t)}\Big\| 
Y(t-\tau,\cdot-y)-Y(t,\cdot)\Big\|_{L^p}|f(\tau,y)|\, d y\, d\tau.
\end{split}
\end{equation}
 If \(p=\infty\), the same estimate holds trivially.
Note that in \(\Omega_1(t)\) we have \(t\ge t-\tau\ge t(1-\delta)\ge\frac12 
t\), so \(t-\tau\) and \(t\) are comparable and there is no singularity in 
\(\tau\). Our aim is to prove that the \(L^p\)-norm on the left-hand side of (\ref{minkow}) tends to \(0\) as 
\(\delta\to 0\) {\em uniformly in \(t\)}. To achieve this, we distinguish two different cases w.r.t.~$x\in \R^d$
when looking at the $L^p$-norm on the right-hand side of (\ref{minkow}):
\begin{itemize}
 \item[(i)] \(|x-y|\le 2(\delta t)^{\alpha/\beta}\),
\item[(ii)] \(|x-y|>2(\delta t)^{\alpha/\beta}\).
\end{itemize}
Observe that this splitting seems to be needed. If we simply estimate the \(L^p\)-norm by 
the triangle inequality
\begin{equation}\label{triangle_est1}
\Big\| 
Y(t-\tau,\cdot-y)-Y(t,\cdot)\Big\|_{L^p}\le \Big\| 
Y(t-\tau,\cdot-y)\Big\|_{L^p}+\Big\|Y(t,\cdot)\Big\|_{L^p}=:I_1+I_2,
\end{equation}
we would get a bound
\[
I_1+I_2\lesssim t^{\alpha-\frac{\alpha 
d}{\beta}(1-\frac{1}{p})-1},
\]
which is of a right form but the problem is that this quantity does not 
converge to zero as \(\delta\to 0\) which is what we are after. Therefore we need to do the estimates more 
carefully.

The motivation for the splitting is that in the case (i) both \(|x-y|\) and 
\(|x|\) are bounded from above by a multiple of \((\delta t)^{\alpha/\beta}\). 
In 
this case we will simply use the triangle inequality~\eqref{triangle_est1}. The 
second case (ii) is more complicated, but here we will proceed as follows. Since 
there are differences both in the space and the time variable, we treat the 
differences separately by using the triangle inequality:
\begin{equation}\label{triangle_est2}
\begin{split}
\Big\| 
Y(t-\tau,\cdot-y)-Y(t,\cdot)\Big\|_{L^p}\le & \,\Big\| 
Y(t-\tau,\cdot-y)-Y(t-\tau,\cdot)\Big\|_{L^p}\\
&+\Big\|Y(t-\tau,\cdot)-Y(t,\cdot)\Big\|_{L^p}\\
=:& \, I_3+I_4.
\end{split}
\end{equation}

In both of these we shall use the Mean Value Theorem. Note that in the case 
(ii) we 
are away from the singularities of \(x\) and \(t\), since \(t-\tau\ge\frac12t\) 
and \(|x|=|x-y+y|\ge |x-y|-|y|\ge (\delta t)^{\alpha/\beta}\).

Since the asymptotic behavior of \(Y\) is different for each \(d\), we 
consider here only 
the case \(d>2\beta\) and \(0<\alpha<1\). The other cases can be treated 
similarly, since the proof is based only on the pointwise 
estimates for \(Y\), $\nabla Y$, and $\partial_t Y$ given in Lemmas~\ref{fund_sol_asympt_Y} and~\ref{G_gradient}.

We start with the case (i). Note that for \((\tau,y)\in\Omega_1(t)\) we have
\[
\frac{|x-y|}{(t-\tau)^{\alpha/\beta}}\le\frac{2\delta^{\alpha/\beta}}{
(1-\delta)^{\alpha/\beta}},
\]
so we may use the asymptotic behavior of Lemma~\ref{fund_sol_asympt_Y} for small values of the similarity 
variable \(R\) to obtain
\begin{equation}\label{Y_asympt1}
|Y(t,x)|\lesssim t^{-\alpha-1}|x|^{-d+2\beta}.
\end{equation}
As mentioned before, we 
use~\eqref{triangle_est1} and~\eqref{Y_asympt1} to obtain
\[
I_1\lesssim t^{-\alpha-1}\Big(\int_{|x-y|\le 2(\delta t)^{\alpha/\beta}} 
|x-y|^{(-d+2\beta)p}\, d x\Big)^{1/p}.
\]
Introducing the spherical coordinates gives the desired estimate
\[
I_1\lesssim 
\delta^{\frac{\alpha}{\beta}(-d+2\beta+\frac{d}{p})}t^{\alpha-\frac{
\alpha 
d}{\beta}(1-\frac{1}{p})-1}.
\]
Notice that the assumption $p \in [1, \kappa_2)$ guarantees the integrability and 
the positivity of the power of \(\delta\), which is needed in the end. The same 
proof applies also for \(I_2\).

Now we shall provide the estimate in the second case (ii). Since we are going 
to use the Mean Value Theorem, we need to calculate the derivatives of the 
fundamental solution \(Y\). We recall the following estimates from Lemma~\ref{G_gradient} for $d > 2\beta$:
\begin{equation}\label{Y_asympt2}
|\nabla Y(t,x)|\lesssim t^{2\alpha-1}|x|^{-\beta-d-1},\quad |x|^\beta 
t^{-\alpha}\ge1,
\end{equation}
and
\begin{equation}\label{Y_asympt21}
|\nabla Y(t,x)|\lesssim t^{-\alpha-1}|x|^{-d-1+2\beta},\quad |x|^\beta 
t^{-\alpha}\le 1.
\end{equation}

By using the Mean Value Theorem for \(I_3\) we obtain
\[
I_3=|y|\| \nabla Y(t-\tau,\tilde{x}(\cdot))\|_{L^p}
\]
for some \(\tilde{x}\) on the line between \(x-y\) and \(x\), where \(x\) 
denotes the integration variable.

Since 
\[
\begin{split}
|\tilde{x}|=|x-y+\tilde{x}-(x-y)|&\ge|x-y|-|\tilde{x}-(x-y)|\\
&\ge |x-y|-|y|\ge\frac{|x-y|}{2},
\end{split}
\]
we have
\begin{equation}\label{tilde_est}
\frac{|\tilde{x}|}{(t-\tau)^{\alpha/\beta}}\ge\frac{|x-y|}{
2t^{\alpha/\beta}}\ge{\delta^{\alpha/\beta}}.
\end{equation}

Notice, that since \(\delta\) can be small, we have to use the asymptotics near 
zero and near infinity. Therefore we divide the integral \(I_3\) into two parts 
\(I_{31}\) and \(I_{32}\) depending on whether \(|\tilde{x}|^\beta 
(t-\tau)^{-\alpha}\) is less than \(1\) or greater than \(1\).

In \(I_{32}\) we use
\[
|\tilde{x}|\le |x-y|+|y|\le\frac32|x-y|,
\]
so the set of integration is contained in the set
\[
\Big\{ x\in\R^d\,:\, |x-y|\ge \frac23(t-\tau)^{\alpha/\beta}\Big\},
\]
which implies the estimate
\[
\begin{split}
I_{32}&\lesssim (\delta t)^{\alpha/\beta}\Big(\int_{|x-y|\ge 
\frac23(t-\tau)^{\alpha/\beta}} 
(t-\tau)^{(2\alpha-1)p}|\tilde{x}|^{(-d-1-\beta)p}\, d x\Big)^{1/p}\\
&\lesssim \delta^{\alpha/\beta}t^{\alpha/\beta+2\alpha-1}\Big(\int_{|x-y|\ge 
\frac23(t-\tau)^{\alpha/\beta}} 
|x-y|^{(-d-1-\beta)p}\, d x\Big)^{1/p}.
\end{split}
\]

Introducing spherical coordinates gives the estimate
\[
I_{32}\lesssim \delta^{\alpha/\beta} t^{\alpha-\frac{\alpha 
d}{\beta}(1-\frac{1}{p})-1},
\]
which is of the form we need.

For \(I_{31}\) we note that by~\eqref{tilde_est} the set of integration is 
contained in the set
\[
\Big\{x\in\R^d\,:\,\delta^{\alpha/\beta}\le 
\frac{|x-y|}{ 2(t-\tau)^ { \alpha/\beta } } \le 1\Big\},
\]
so by using~\eqref{Y_asympt21} we obtain
\[
I_{31}\lesssim (\delta t)^{\alpha/\beta}\Big(\int_{\delta^{\alpha/\beta}\le 
\frac{|x-y|}{ 2(t-\tau)^ { \alpha/\beta } } \le 1} 
(t-\tau)^{(-\alpha-1)p}|\tilde{x}|^{(-d-1+2\beta)p}\, d x\Big)^{1/p}.
\]
Once again we use the fact that \(|\tilde{x}|\) and 
\(|x-y|\) are comparable. We may proceed as before except we have to separate 
two cases: (a) \((-d-1+2\beta)p=-d\) or (b) \((-d-1+2\beta)p\neq-d\). An easy 
calculation shows that the first case is possible only in the case 
\(\beta\ge\frac12\). The case (a) leads to a logarithmic function. Indeed, we 
may estimate
\[
\begin{split}
I_{31}&\lesssim 
\delta^{\alpha/\beta}t^{\alpha/\beta-\alpha-1}\Big(\int_{\delta^{\alpha/\beta}
\le 
\frac{|x-y|}{ 2(t-\tau)^ { \alpha/\beta } } \le 1} 
|x-y|^{(-d-1+2\beta)p}\, d 
x\Big)^{1/p}\\
&\lesssim\delta^{\alpha/\beta}\big|\log\delta 
\big|^{1/p}t^{\alpha/\beta-\alpha-1}.
\end{split}
\]

A simple arithmetic calculation shows that the power of \(t\) is actually
\[
\frac{\alpha}{\beta}-\alpha-1=\alpha-\frac{\alpha d}{\beta}(1-\frac{1}{p})-1,
\]
which is exactly of the right form and the factor depending on \(\delta\) tends 
to zero as \(\delta\to 0\) uniformly in \(t\).

The assumption $p \in [1, \kappa_2)$ leads to a usual power function similarly as before. We 
omit the details and write the final estimate
\begin{equation}\label{I_31_est_final}
I_{31}\lesssim 
\Big|\delta^{\frac{\alpha}{\beta}}-\delta^{2\alpha-\frac{\alpha 
d}{\beta}(1-\frac{1}{p})}\Big|t^{\alpha-\frac{\alpha 
d}{\beta}(1-\frac{1}{p})-1}.
\end{equation}
Again the assumption $p \in [1, \kappa_2)$ guarantees that the second power of \(\delta\) 
is positive, so we have obtained the desired estimate also in this case.

For \(I_4\) we use again the Mean Value Theorem to obtain
\[
I_4=\tau\|\partial_t Y(\tilde{t},\cdot)\|_{L^p}
\]
for some \(\tilde{t}\in(t-\tau,t)\). Note that in \(\Omega_1(t)\) \(t\) and 
\(\tilde{t}\) are comparable: \((1-\delta)t\le\tilde{t}\le t\).

Now \(|x|=|x-y+y|\ge |x-y|-|y|\ge (\delta t)^{\alpha/\beta}\), so 
\begin{equation}\label{tilde_z_est}
\tilde{z}:=\frac{|x|}{\tilde{t}^{\alpha/\beta}}\ge\frac{|x|}{t^{
\alpha/\beta}}\ge\frac { (\delta 
t)^{\alpha/\beta}}{t^{\alpha/\beta}}=\delta^{\alpha/\beta}
\end{equation}
and again we have two cases, since \(\delta\) can be small. We denote the 
integrals by \(I_{41}\) and \(I_{42}\) depending on whether \(\tilde{z}\le 1\) 
or \(\tilde{z}\ge 1\). 

Again, we recall from Lemma~\ref{G_gradient} the estimates
\begin{equation}\label{Y_der_t_est1}
|\partial_t Y(t,x)|\lesssim t^{-\alpha-2}|x|^{-d+2\beta},\quad 
\frac{|x|^\beta}{t^\alpha}\le1,
\end{equation}
and
\begin{equation}\label{Y_der_t_est2}
|\partial_t Y(t,x)|\lesssim t^{2\alpha-2}|x|^{-d-\beta},\quad 
\frac{|x|^{\beta}}{t^\alpha}\ge 1.
\end{equation}

The estimates~\eqref{tilde_z_est} and~\eqref{Y_der_t_est1} now give for \(I_{41}\) that
\[
I_{41}\lesssim \delta t\Big(\int_{t^{\alpha/\beta}\ge |x|\ge 
(\delta t)^{\alpha/\beta}} \tilde{t}^{(-\alpha-2)p}|x|^{(-d+2\beta)p}\, d 
x\Big)^{1/p}.
\]
By changing the 
variables \(x\leftrightarrow \frac{x}{t^{\alpha/\beta}}=:z\), we 
obtain
\[
\begin{split}
I_{41}&\lesssim \delta t^{\alpha-\frac{\alpha 
d}{\beta}(1-\frac{1}{p})-1}\Big(\int_{\delta^{\alpha/\beta}\le |z|\le 
1} |z|^{(-d+2\beta)p}\, d z\Big)^{1/p}\\
&\lesssim 
\Big|\delta-\delta^{1+2\alpha-\frac{\alpha d}{\beta}(1-\frac{1}{p})}\Big| 
t^{\alpha-\frac{\alpha 
d}{\beta}(1-\frac{1}{p})-1}.
\end{split}
\]
Since the powers of \(\delta\) are even better than in~\eqref{I_31_est_final}, 
we have derived the desired estimate for \(I_{41}\).

For \(I_{42}\) we observe that
\[
1\le \tilde{z}\le\frac{|x|}{((1-\delta)t)^{\alpha/\beta}}
\]
which implies
\[
 |x|\ge 
((1-\delta)t)^{\alpha/\beta}.
\]
We use~\eqref{Y_der_t_est2} to obtain
\[
I_{42}\lesssim \delta 
t\Big(\int_{\frac{|x|}{t^{\alpha/\beta}}\ge(1-\delta)^{\alpha/\beta}} 
\tilde{t}^{(2\alpha-2)p}|x|^{(-d-\beta)p}\, d x\Big)^{1/p}.
\]
Making the obvious change of variables \(x\leftrightarrow 
\frac{x}{t^{\alpha/\beta}}=:z\) we end up with the estimate
\[
I_{42}\lesssim \delta t^{\alpha-\frac{\alpha 
d}{\beta}(1-\frac{1}{p})-1}
\]
similarly as before.



Collecting all above we see that
\[
t^{1+\frac{\alpha 
d}{\beta}(1-\frac{1}{p})-\alpha}\Big\|\iint_{\Omega_1(t)} 
(Y(t-\tau,\cdot-y)-Y(t,\cdot))f(\tau,y)\, d 
y\, d\tau\Big\|_{L^p}\lesssim\delta^{\eta}
\|f\|_1
\]
for some positive number \(\eta\). The upper bound tends to zero as \(\delta\to 
0\) {\em uniformly in \(t\)}.

We now fix \(\delta_0<\frac12\) such that the previous term is small and 
continue to estimate the norm
\[
t^{1+\frac{\alpha 
d}{\beta}(1-\frac{1}{p})-\alpha}\Big\|\iint_{\Omega_2(t)} 
(Y(t-\tau,\cdot-y)-Y(t,\cdot))f(\tau,y)\, d 
y\, d\tau\Big\|_{L^p}.
\]

Using the integral form of the Minkowsky inequality we have
\[
\begin{split}
&t^{1+\frac{\alpha 
d}{\beta}(1-\frac{1}{p})-\alpha}\Big\|\iint_{\Omega_2(t)} 
(Y(t-\tau,\cdot-y)-Y(t,\cdot))f(\tau,y)\, d y\, d\tau\Big\|_{L^p}\\
\le&t^{1+\frac{\alpha 
d}{\beta}(1-\frac{1}{p})-\alpha}\iint_{\Omega_2(t)}\Big\| 
Y(t-\tau,\cdot-y)\Big\|_{L^p}|f(\tau,y)|\, d y\, d\tau\\
&+t^{1+\frac{\alpha 
d}{\beta}(1-\frac{1}{p})-\alpha}\iint_{\Omega_2(t)}\Big\| 
Y(t,\cdot)\Big\|_{L^p}|f(\tau,y)|\, d y\, d\tau\\
=:&I_5+I_6
\end{split}
\]

By Lemma~\ref{Y_decay_estimate} we have that\(\|Y(t,\cdot)\|_{L^p}\sim 
t^{\alpha-\frac{\alpha 
d}{\beta}(1-\frac{1}{p})-1}\) and, therefore, we may directly estimate \(I_6\) by
\[
I_6\lesssim \iint_{\Omega_2(t)}|f(\tau,y)|\, d y\, d\tau\to 0,
\]
as \(t\to\infty\). 

For \(I_5\) we have two possibilities: either \(\tau\le \delta_0 t\) or 
\(\tau\ge \delta_0t\). According to this we split the domain 
\(\Omega_2^{(0)}(t)\) into two parts:
\[
\Omega_2^{(0)}(t)=(0,\delta_0 t)\times \{y\in\R^d\,:\, |y|\ge (\delta 
t)^{\alpha/\beta}\}\cup (\delta_0 t,t)\times\R^d,
\]
where \((0)\) indicates the fact that we have fixed \(\delta=\delta_0\). 

Hence, \(I_5\) can be written as
\[
\begin{split}
I_5=&t^{1+\frac{\alpha 
d}{\beta}(1-\frac{1}{p})-\alpha}\int_{0}^{\delta_0 t}\int_{|y|\ge(\delta_0 
t)^{\alpha/\beta}} 
\|Y(t-\tau,\cdot-y)\|_{L^p}|f(\tau,y)|\, d y\, d\tau\\
&+t^{1+\frac{\alpha 
d}{\beta}(1-\frac{1}{p})-\alpha}\int_{\delta_0 t}^{t}\int_{\R^d} 
\|Y(t-\tau,\cdot-y)\|_{L^p}|f(\tau,y)|\, d y\, d\tau.
\end{split}
\]

We use the same bound \(\|Y(t,\cdot)\|_{L^p}\lesssim t^{\alpha-\frac{\alpha 
d}{\beta}(1-\frac{1}{p})-1}\) as above for both integrals. Then the first 
integral is dominated by
\[
\begin{split}
&t^{1+\frac{\alpha 
d}{\beta}(1-\frac{1}{p})-\alpha}\int_{0}^{\delta_0 t}\int_{|y|\ge (\delta_0 
t)^{\alpha/\beta}}(t-\tau)^{\alpha-\frac{\alpha 
d}{\beta}(1-\frac{1}{p})-1}|f(\tau,y)|\, d\tau\, d y\\
\le &(1-\delta_0)^{\alpha-\frac{\alpha 
d}{\beta}(1-\frac{1}{p})-1}\int_0^{\delta_0 t}\int_{|y|\ge (\delta_0 
t)^{\alpha/\beta}}|f(\tau,y)|\, d\tau\, d y,
\end{split}
\]
which clearly tends to zero as \(t\to\infty\). The upper bound for the second 
integral is
\[
t^{1+\frac{\alpha 
d}{\beta}(1-\frac{1}{p})-\alpha}\int_{\delta_0 
t}^t\int_{\R^d}(t-\tau)^{\alpha-\frac{\alpha 
d}{\beta}(1-\frac{1}{p})-1}|f(\tau,y)|\, d\tau\, d y
\]
 
This integral causes problems, since now there is a singularity in \(t\). But the 
assumption $p \in [1, \kappa_2)$ guarantees that the singularity is weak.
We use the decay condition~\eqref{decay_cond_f} imposed for the source 
term. By using this, we have
\[
\begin{split}
I_5\lesssim& \,
t^{1+\frac{\alpha 
d}{\beta}(1-\frac{1}{p})-\alpha}\int_{\delta_0 
t}^t(t-\tau)^{\alpha-\frac{\alpha 
d}{\beta}(1-\frac{1}{p})-1}(1+\tau)^{-\gamma}\, d\tau\\
\lesssim& \,
t^{1+\frac{\alpha 
d}{\beta}(1-\frac{1}{p})-\alpha-\gamma}\int_{0}^{\delta_0 
t}\tau^{\alpha-\frac{\alpha 
d}{\beta}(1-\frac{1}{p})-1}\, d\tau\lesssim t^{1-\gamma}, 
\end{split}
\]
which tends to zero as \(t\to\infty\), since \(\gamma>1\). This, finally, 
finishes the proof of the case where $d > 2 \beta$ and $0 <\alpha <1$. The other cases are proved similarly. We omit the details.
\end{proof}


\section{Optimal $L^2$-decay for mild solutions}\label{L2}

In this section we will give the proof of Theorem~\ref{thm_optimal_decay}. Here we only consider equation~\eqref{classical_equation}, but our reasoning can be extended to cover a wider range of equations. The main tool we use is Plancherel's theorem, but in general it can be replaced by more general multiplier theorems which allow one to study more general equations, too. For details of such an approach we refer to our earlier paper~\cite{KempSiljVergZach14}. Here we restrict our study to equation~\eqref{classical_equation} for simplified exposition. 

We begin this section by showing that our decay rate is optimal. Indeed, we have the following result.

\begin{proposition}
Let $\alpha\in (0,1)$, \(d\ge 1\), and \(d\neq 2\beta\). Suppose $u$ is the mild solution of the Cauchy problem~\eqref{classical_equation} with $f \equiv 0$. Assume further that \(u_0\in L^1(\R^d)\cap 
L^2(\R^d)\) with $\int_{\Rn} u_0 \, dx \neq0$. Then
\[
\|u(t,\cdot)\|_2\gtrsim t^{-\alpha\mathrm{min}\{1,\frac{d}{2\beta}\}},\quad t\ge 1.
\]
The constant in the estimate depends on $\int_{\Rn} u_0 \, dx$.
\end{proposition}

\begin{proof}

Let \(\rho_0>0\), \(t>0\) and 
\(\rho=\rho(t)\in(0,\rho_0]\). By Plancherel's Theorem, monotonicity of 
\(E_{\alpha,1}\), and the estimate $E_{\alpha,1}(-x)\ge c_1/(1+x)$ for all $x\ge 0$ (with some $c_1>0$), we have
\begin{align}
\|u(t,\cdot)\|_{L^2}^2&=\|\widehat{u}(t,\cdot)\|_{L^2}^2=\int_{\R^d} 
|\widehat{Z}(t,\xi)|^2|\widehat{u}_0(\xi)|^2\, d\xi\nonumber\\
&
\ge\frac{1}{(2\pi)^d} \int_{B_\rho(0)}E_{\alpha,1}(-|\xi|^\beta t^\alpha)^2 
|\widehat{u}_0(\xi)|^2 d \xi \nonumber\\
& \ge \frac{c_1^2}{(2\pi)^d (1+\rho^\beta 
t^\alpha)^2}\int_{B_\rho(0)}|\widehat{u}_0(\xi)|^2\, d\xi\nonumber \\
&=\frac{c_2}{(1+\rho^\beta 
t^\alpha)^2}\rho^d\Big(\rho^{-d}\int_{B_\rho}|\widehat{u}_0(\xi)|^2\, d\xi\Big). \label{lest1}
\end{align}

By the Plancherel Theorem and the Riemann-Lebesgue Lemma we have 
\(\widehat{u}_0\in C_0(\R^d)\cap L^2(\R^d)\). By the Lebesgue differentiation 
theorem, we may choose \(\rho_0\)  small enough in order to obtain
\[
\rho^{-d}\int_{B_\rho} |\widehat{u}_0(\xi)|^2\, d\xi\ge \frac{|\widehat u(0)|^2}2\quad\textrm{for 
all \(\rho\in(0,\rho_0]\)}.
\]
Using this in~\eqref{lest1} gives the lower bound 

\begin{equation}\label{lb1}
\|u(t,\cdot)\|_{L^2}^2\ge \frac{c_2|\widehat u(0)|^2\rho^d}{2(1+\rho^\beta t^\alpha)^2}.
\end{equation}

Next we choose 
\(\rho=\rho_0\), which yields
\[
\|u(t,\cdot)\|_{L^2}^2
\gtrsim t^{-2\alpha}
\]
for \(t\ge 1\). On the other hand, the choice \(\rho=\rho(t)=\frac{\rho_0}{(1+t^\alpha)^{1/\beta}}\) gives \(\rho(t)^\beta t^\alpha\le \rho_0^\beta\) and thus 
by~\eqref{lb1} we get the estimate
\[
\|u(t,\cdot)\|_{L^2}^2
\gtrsim t^{-\frac{\alpha 
d}{\beta}},\quad t\ge 1.
\]
These estimates combined together give the claimed lower bound. 
\end{proof}

Observe that the constant in the above proposition is of the form $C=C(\rho_0)|\int_{\Rn} u \, dx|$, where also $\rho_0$ depends on $|\int_{\Rn} u \, dx|$. Nevertheless, we obtain that the decay rate in Theorem~\ref{thm_optimal_decay} is optimal. We will now give a proof of this decay result.

\begin{proof}[Proof of Theorem~\ref{thm_optimal_decay}]
To prove the upper bound, we proceed as in~\cite[Theorem 4.2]{KempSiljVergZach14}. Suppose 
that \(d<2\beta\). By Plancherel's Theorem, the Riemann-Lebesgue Lemma and the 
estimate~\eqref{ML2_estimate}, we have
\begin{equation}\label{uest1}
\begin{split}
\|u(t,\cdot)\|_{L^2}^2&=\|\widehat{u}(t,\cdot)\|_{L^2}^2=\int_{\R^d}|\widehat{Z}(t,
\xi)|^2|\widehat{u}_0(\xi)|^2\, d\xi\le \|\widehat{u}_0\|_{L^\infty}\int_{\R^d}
|\widehat{Z}(t,\xi)|^2\, d\xi\\
&\lesssim \|u_0\|_{L^1}^2\int_{\R^d}\frac{\, d\xi}{(1+|\xi|^\beta 
t^\alpha)^2}=\|u_0\|_{L^1}^2t^{-\frac{\alpha 
d}{\beta}}\int_{\R^d}\frac{\, d\eta}{(1+|\eta|^\beta)^2},
\end{split}
\end{equation}
where in the last step we have made the change of variables 
\(\xi\leftrightarrow \xi t^{\alpha/\beta}=:\eta\). Now the condition 
\(d<2\beta\) guarantees that the last integral is converging. Hence we have 
derived the upper bound in the case \(d<2\beta\).

We are left with the case \(d > 2\beta\). Here we use the Hardy-Littlewood-Sobolev
Theorem on fractional integration. Indeed, we choose \(q=2\) 
in Theorem~\ref{HLS_thm} to obtain
\begin{equation}\label{HLS_est1}
\|(-\Delta)^{-\frac{\beta}{2}} u_0\|_{L^2}\lesssim \|u_0\|_{L^\frac{2d}{d+\beta}}<\infty,
\end{equation}
since \(u_0\in L^1(\R^d)\cap L^2(\R^d)\) implies that \(u_0\in 
L^{\frac{2d}{d+\beta}}(\R^d)\) by interpolation. Using this and the 
estimate~\ref{ML_estimate}, we have
\[
\begin{split}
\|u(t,\cdot)\|_{L^2}^2&= \int_{\R^d} 
|\xi|^{2\beta}|\widehat{Z}(t,\xi)|^2||\xi|^{-\beta}\widehat{u}
_0(\xi)|^2\, d\xi\\
&\lesssim t^{-2\alpha}\int_{\R^d}\frac{|\xi|^{2\beta} 
t^{2\alpha}}{(1+|\xi|^{2\beta}t^{2\alpha})^2}||\xi|^{-\beta}\widehat{u}
_0(\xi)|^2\, d\xi\\
&\lesssim t^{-2\alpha} 
\int_{\R^d}|\xi|^{-2\beta}|\widehat{u}_0(\xi)|^2\, d\xi=t^{-2\alpha}\|(-\Delta)^{
-\frac{\beta}{2}} u_0\|_{L^2}^2,
\end{split}
\]
which completes the proof by~\eqref{HLS_est1}.

For the borderline case $d=2\beta$ we estimate directly by Young's inequality~\eqref{Youngw} to obtain
\[
\|u(t, \cdot)\|_{L^{2,\infty}}=\|u_0(\cdot) \star Z(t, \cdot)\|_{L^{2,\infty}} \le C \|Z(t, \cdot)\|_{L^{2, \infty}}\|u_0\|_{L^1} \le C\|u_0\|_{L^1}   t^{-\alpha},
\]
where we used Lemma~\ref{Z_Lp_decay} to estimate the weak $L^2$--norm of $Z$. This finishes the proof.
\end{proof}

\section{Energy method and $L^2$--decay for weak solutions}

In this section we consider the $L^2$--decay of weak solutions which are defined in Definition~\ref{weak_solution_def}.  We will restrict our study to the homogeneous case $f \equiv 0$. We will proceed in a rather formal manner where we prove the estimates starting directly from the equation by multiplying it with the appropriate test functions. For the details required for the rigorous treatment starting from the Definition~\ref{weak_solution_def}, we refer to~\cite{KempSiljVergZach14}.

In the proof of Theorem~\ref{decay_weak_solution}, we will need the following Lemma from~\cite{VergZach15}.

\begin{lemma} \label{LemmaL2IN}
Let $T>0$ and $\Omega\subset \Rn$ be an open set. Let $k\in W^{1,1}_{loc}([0,\infty))$ be nonnegative and nonincreasing. Then for any $v\in L^2((0,T)\times \Omega)$ and any $v_0\in L^2(\Omega)$ there holds
\begin{equation} \label{L2norminequ}
\int_{\Omega}v\partial_t\big(k * [v-v_0]\big)\,dx\ge \|v(t, \cdot)\|_{L^2(\Omega)}
\partial_t\big(k* \big[\|v\|_{L^2(\Omega)}-\|v_0\|_{L^2(\Omega)}\big]\big)(t),
\end{equation}
for almost every $t\in (0,T)$.
\end{lemma}
\begin{proof}
The result is originally from~\cite{VergZach15}. For the proof in our context we refer to Lemma 6.2 in~\cite{KempSiljVergZach14}.
\end{proof}

Observe that in our case the kernel $k$ corresponds to $g_{1-\alpha}$. The function $g_{1-\alpha}$ is, however, not in $W^{1,1}$. For this reason, a rigorous treatment of the problem requires an appropriate regularization of the fractional
derivation operator in time. One way to do this is via its Yosida approximations, which leads to an integro-differential
operator of the same form with a kernel $g_{1-\alpha, n}$ that is also nonnegative and nonincreasing, and which belongs
to $W^{1,1}_{loc}([0,\infty))$. The details of such calculations can be found in~\cite{KempSiljVergZach14}, see also 
\cite{Zach08}.
Note that the regularized weak formulation used in~\cite{KempSiljVergZach14} and \cite{Zach08} does not involve an integral in time
on $[0,T]$, but it requires the validity of a certain relation pointwise a.e.\ in $(0,T)$. Here we proceed on a formal level by using the singular kernel $g_{1-\alpha}$ and a formulation of the problem where we only integrate in space (not in time) against a test
function.

\begin{lemma}\label{L1}
Let $u_0\in L^1(\R^d)\cap L^2(\R^d)$. Suppose $u$ is a weak solution of equation~\eqref{weak_equation} with initial condition $u|_{t=0}=u_0$, and assume that
(\ref{ucondition}) is satisfied.
Then 
\[
\|u(t,\cdot)\|_{L^1(\Rn)} \le  \|u_0\|_{L^1}
\]
for a.a.\ $t > 0$.
\end{lemma}

\begin{proof}
Letting $R>0$ we choose a nonnegative cut-off function $\psi\in C^1_0(B_{R+1})$ such that $\psi=1$ in $B_{R}$ and
$\psi\le 1$ as well as $|\nabla \psi|\le 2$ in $B_{R+1}$. Here $B_\rho$ denotes the ball of radius $\rho>0$ and center $0$. 
For $\varepsilon>0$, define
\[
H_\varepsilon(y)=(y^2+\varepsilon^2)^{\frac{1}{2}}-\varepsilon,\quad y\in \R.
\]
Clearly $H_\varepsilon\in C^1(\R)$ and $H'_\varepsilon\in W^1_\infty(\R)$. Indeed,
\[
H_\varepsilon'(y)=\frac{y}{(y^2+\varepsilon^2)^{\frac{1}{2}}},\quad y\in \R.
\]
Observe that $H_\varepsilon$ is convex. Testing the PDE with $H_\varepsilon'(u)\psi$ gives
\begin{align*}
&\int_{\Rn} H_\varepsilon'(u)\psi \partial_t^\alpha (u-u_0)\,dx+ F_\varepsilon(t)=0,\quad t>0,
\end{align*}
where
\[
F_\varepsilon(t)= \int_{\R^d}\int_{\R^d}
 K(x,y)[u(t,x)-u(t,y)] \big[H_\varepsilon'(u(t,x))\psi(x)-H_\varepsilon'(u(t,y))\psi(y)\big]\, dx \, dy.  
\]
Since $H_\varepsilon$ is convex, we may (formally) use the inequality from Corollary 6.1 in \cite{KempSiljVergZach14},
to the result that pointwise a.e. we have
\[ 
H_\varepsilon'(u)\partial_t^\alpha (u-u_0)\ge \partial_t^\alpha \big(H_\varepsilon(u)-H_\varepsilon(u_0)\big).
\]
Applying this to the previous relation and convolving the resulting inequality with $g_\alpha$ we obtain
\[
\int_{\Rn}\big(H_\varepsilon(u)-H_\varepsilon(u_0)\big)\psi\,dx+g_\alpha\ast F_\varepsilon \le 0,\quad t>0.
\]
Next, we send $\varepsilon\to 0$ and observe that $H_\varepsilon(y)\to |y|$ as well as $H_\varepsilon'(y)\to
 \text{sign\,}y$ for $y\in \R$. Thus we get
\begin{equation} \label{Fequ}
\int_{\Rn}\big(|u(t,x)|-|u_0(x)|\big)\psi(x)\,dx+(g_\alpha\ast F)(t) \le 0,\quad t>0,
\end{equation}
with
\begin{align*}
F(t) &=\int_{\Rn}\int_{\R^d} K(x,y)[u(t,x)-u(t,y)]\cdot [\text{sign\,}u(t,x)\psi(x)-\text{sign\,}u(t,y)\psi(y)]\, dx \, dy\\
& \ge \int_{\Rn}\int_{\R^d} K(x,y)[u(t,x)-u(t,y)]\cdot [\psi(x)-\psi(y)]\,\text{sign\,}u(t,y) \, dx \, dy=:F_1(t).
\end{align*}

Using the properties of $\psi$ and $K$, we may estimate as follows.
\begin{align*}
|F_1(t)| & \le \Lambda \int_{\Rn}\int_{\R^d} \frac{|u(t,x+h)-u(t,x)|\cdot |\psi(x+h)-\psi(x)|}{|h|^{d+\beta}} \, dh \, dx\\
& =  \Lambda \int_{\Rn}\int_{|h|>R/2}\ldots dh\,dx+ \Lambda \int_{\R^d\setminus B_{R/2}}\int_{|h|\le R/2}\ldots dh\,dx\\
& \le \frac{2\Lambda}{(R/2)^\frac{\beta}{2}} \int_{\Rn}\int_{|h|>R/2} \frac{|u(t,x+h)-u(t,x)|}{|h|^{d+\frac{\beta}{2}}}\,dh\,dx\\
& \quad +\Lambda \int_{\R^d\setminus B_{R/2}}\int_{|h|\le R/2}\frac{|u(t,x+h)-u(t,x)|\cdot 2^{1-\frac{\beta}{2}}
|\nabla \psi|_\infty^\frac{\beta}{2}}{|h|^{d+\frac{\beta}{2}}}\, dh\,dx\\
& \le  \frac{2^{1+\frac{\beta}{2}}\Lambda}{R^\frac{\beta}{2}}\,[u(t,\cdot)]_{W^{\frac{\beta}{2},1}(\R^d)}
+2\Lambda \int_{\R^d\setminus B_{R/2}}\int_{\R^d}\frac{|u(t,x+h)-u(t,x)|}{|h|^{d+\frac{\beta}{2}}}\, dh\,dx,
\end{align*}
where the last two terms tend to $0$ for a.a.\ $t>0$ as $R\to \infty$, by assumption (\ref{ucondition}). Thus the assertion follows from (\ref{Fequ}) and the previous estimates by sending $R\to \infty$.
\end{proof}

Finally, we have the following Lemma, which shows that the \(L^2\)-norm of a 
weak solution is a subsolution to a purely time-fractional equation.

\begin{lemma}
Let $K$, $u_0$, and $u$ be as in the previous lemma. Then there exists a constant $\mu=\mu(d,\beta, \lambda, \|u_0\|_{L^1}) >0$ such that (formally)
\begin{equation}\label{time_fractional}
\partial_t^\alpha \big[\|u(t, \cdot)\|_{L^2}-\|u_0\|_{L^2}\big]+\mu\|u(t, \cdot)\|_{L^2}^{1+\frac{2\beta}{d}} \le 0, \quad t \ge 0.
\end{equation}
\end{lemma}

\begin{proof}
Choose the test function $\varphi=u$ in Definition~\ref{weak_solution_def} of weak solutions. We apply Lemma~\ref{LemmaL2IN} for the fractional time derivative to obtain
\[
\|u(t, \cdot)\|_{L^2} \partial_t^\alpha \big[\|u(t, \cdot)\|_{L^2}-\|u_0\|_{L^2}\big]+ \int_{\Rn}\int_{\Rn} K(x,y)[u(t,x)-u(t,y)]^2 \, dx \, dy \le 0.
\]
For the elliptic term, we then use the fractional Nash inequality (cf.~\cite[p. 13]{TerERobiSikoZhu07}) together with the assumption~\eqref{K_assumption} on the kernel $K$ and with Lemma~\ref{L1} to obtain 
\begin{align*}
\|u(t, \cdot)\|_{L^2}^{2+\frac{2\beta}{d}} &\le C\|u(t, \cdot)\|_{L^1}^{\frac{2\beta}{d}}\int_{\Rn}\int_{\Rn} \frac{[u(t,x)-u(t,y)]^2}{|x-y|^{d+\beta}} \, dx \, dy  \\
&\le C\|u_0\|_{L^1}^{\frac{2\beta}{d}}\int_{\Rn}\int_{\Rn} K(x,y)[u(t,x)-u(t,y)]^2 \, dx \, dy.
\end{align*}
This concludes the proof.
\end{proof}

\subsection{Proof of Theorem~\ref{decay_weak_solution}}
We are finally ready to prove the decay result for the weak solutions. The proof is based on using the comparison principle for the purely time-fractional equation~\eqref{time_fractional}.

\begin{proof}[Proof of Theorem~\ref{decay_weak_solution}]
Let $T >0$ be an arbitrary real number. By the comparison principle for time-fractional differential
equations (see \cite[Lemma 2.6 and Remark 2.1]{VergZach15}), the inequality~\eqref{time_fractional} implies that $\|u(t, \cdot)\|_{L^2}\le w(t)$ for almost every \
$t\in (0,T)$, where $w$ solves the equation corresponding to~\eqref{time_fractional}, that is
\[
\partial_t^\alpha (w-w_0)(t)+\mu w(t)^\gamma= 0,\quad t>0,\quad w(0)=w_0:= \|u_0\|_{L^2},
\]
where we put $\gamma={1+\frac{2\beta}{d}}$. It is known that for $w_0>0$ there exist constants $c_1, c_2>0$ such that
\[
\frac{c_1}{1+t^{\frac{\alpha}{\gamma}}}\,\le w(t)\le\,\frac{c_2}{1+t^{\frac{\alpha}{\gamma}}},\quad t\ge 0,
\]
see \cite[Theorem 7.1]{VergZach15}. Since $T>0$ was arbitrary, we conclude that
\[
\|u(t, \cdot)\|_{L^2} \le \,w(t)\le \,\frac{c_2}{1+t^{\frac{\alpha}{\gamma}}}\,=\,\frac{c_2}{1+t^{\frac{\alpha d}{d+2\beta}}},\quad
\mbox{almost every}\;t>0.
\]
This finishes the proof of Theorem~\ref{decay_weak_solution}.
\end{proof}

{\bf Acknowledgements.}
The work has been partially conducted during the research visits of R.Z. to Aalto University in Spring 2013, and of J.K. and J.S. to the University of Ulm in Fall 2014; we thank both institutions for their kind hospitality.
 
The research visits of J.S. has been supported by a V\"ais\"al\"a foundation travel grant as well as via the Ulm International Research Fellows in Mathematics and Economics program of the Faculty of Mathematics and Economics of Ulm university. In addition, J.S. has enjoyed financial support from the Academy of Finland grant 259363, and R.Z., in turn, has been supported by a Heisenberg fellowship of the German Research Foundation (DFG), GZ Za 547/3-1.

\def\cprime{$'$} \def\cprime{$'$}

\bibliographystyle{plain}

\begin{thebibliography}{10}

\bibitem{AlleCaffVass15}
Mark Allen, Luis Caffarelli, and Alexis Vasseur.
\newblock A parabolic problem with a fractional-time derivative.
\newblock {\em arXiv preprint arXiv:1501.07211}, 2015.

\bibitem{BarlBassChenKass09}
Martin~T. Barlow, Richard~F. Bass, Zhen-Qing Chen, and Moritz Kassmann.
\newblock Non-local {D}irichlet forms and symmetric jump processes.
\newblock {\em Trans. Amer. Math. Soc.}, 361(4):1963--1999, 2009.

\bibitem{Boch12}
Salomon Bochner.
\newblock {\em Harmonic analysis and the theory of probability}.
\newblock Courier Corporation, 2012.

\bibitem{BonfVazq13}
Matteo Bonforte and Juan~Luis Vazquez.
\newblock A priori estimates for fractional nonlinear degenerate diffusion
  equations on bounded domains.
\newblock Preprint, 2013.

\bibitem{BonfVazq14}
Matteo Bonforte and Juan~Luis V{\'a}zquez.
\newblock Quantitative local and global a priori estimates for fractional
  nonlinear diffusion equations.
\newblock {\em Adv. Math.}, 250:242--284, 2014.

\bibitem{Braa36}
Boele Lieuwe~Jan Braaksma.
\newblock Asymptotic expansions and analytic continuations for a class of
  barnes-integrals.
\newblock {\em Compositio Mathematica}, 15:239--341, 1936.

\bibitem{CaffChanVass11}
Luis Caffarelli, Chi~Hin Chan, and Alexis Vasseur.
\newblock Regularity theory for parabolic nonlinear integral operators.
\newblock {\em J. Amer. Math. Soc.}, 24(3):849--869, 2011.

\bibitem{CaffSilv07}
Luis Caffarelli and Luis Silvestre.
\newblock An extension problem related to the fractional {L}aplacian.
\newblock {\em Comm. Partial Differential Equations}, 32(7-9):1245--1260, 2007.

\bibitem{CaffSilv09}
Luis Caffarelli and Luis Silvestre.
\newblock Regularity theory for fully nonlinear integro-differential equations.
\newblock {\em Comm. Pure Appl. Math.}, 62(5):597--638, 2009.

\bibitem{CaffSilv11}
Luis Caffarelli and Luis Silvestre.
\newblock Regularity results for nonlocal equations by approximation.
\newblock {\em Arch. Ration. Mech. Anal.}, 200(1):59--88, 2011.

\bibitem{CaffSilv12}
Luis Caffarelli and Luis Silvestre.
\newblock H\"older regularity for generalized master equations with rough
  kernels.
\newblock Preprint., 2012.

\bibitem{CartCast07}
\'Alvaro Cartea and Diego del Castillo-Negrete.
\newblock Fluid limit of the continuous-time random walk with general l\'evy
  jump distribution functions.
\newblock {\em Phys. Rev. E}, 76:041105, Oct 2007.

\bibitem{ChasChavRoss06}
Emmanuel Chasseigne, Manuela Chaves, and Julio~D Rossi.
\newblock Asymptotic behavior for nonlocal diffusion equations.
\newblock {\em Journal de math{\'e}matiques pures et appliqu{\'e}es},
  86(3):271--291, 2006.

\bibitem{ChenLiYama15}
Xing Cheng, Zhiyuan Li and Masahiro Yamamoto.
\newblock Asymptotic behavior of solutions to space-time fractional diffusion equations
\newblock {\em arXiv preprint arXiv:1505.06965v2}, 2015.

\bibitem{CompCace98}
Albert Compte and Manuel~O. C\'aceres.
\newblock Fractional dynamics in random velocity fields.
\newblock {\em Phys. Rev. Lett.}, 81:3140--3143, Oct 1998.

\bibitem{ContTank04}
Rama Cont and Peter Tankov.
\newblock {\em Financial modelling with jump processes}.
\newblock Chapman \& Hall/CRC Financial Mathematics Series. Chapman \&
  Hall/CRC, Boca Raton, FL, 2004.

\bibitem{DragKlaf00}
Julia Dr\"ager and Joseph Klafter.
\newblock Strong anomaly in diffusion generated by iterated maps.
\newblock {\em Phys. Rev. Lett.}, 84:5998--6001, Jun 2000.

\bibitem{Duan05}
Jun-Sheng Duan.
\newblock Time-and space-fractional partial differential equations.
\newblock {\em Journal of mathematical physics}, 46(1):13504--13504, 2005.

\bibitem{DuoaZuaz92}
Javier Duoandikoetxea and Enrique Zuazua.
\newblock Moments, masses de dirac et d{\'e}composition de fonctions.
\newblock {\em Comptes rendus de l'Acad{\'e}mie des sciences. S{\'e}rie 1,
  Math{\'e}matique}, 315(6):693--698, 1992.

\bibitem{EideKoch04}
Samuil~D. Eidelman and Anatoly~N. Kochubei.
\newblock Cauchy problem for fractional diffusion equations.
\newblock {\em J. Differential Equations}, 199(2):211--255, 2004.

\bibitem{Erdelyi53}
Arthur Erd{\'e}lyi, Wilhelm Magnus, Fritz Oberhettinger, Francesco~G Tricomi,
  and Harry Bateman.
\newblock {\em Higher transcendental functions}, volume~1.
\newblock McGraw-Hill New York, 1953.

\bibitem{FelsKass13}
Matthieu Felsinger, Moritz Kassmann.
\newblock Local regularity for parabolic nonlocal operators.
\newblock {\em Commun. Partial Differ. Equations}, 38:1539--1573, 2013.

\bibitem{GelfShil64} Israel~M. Gelfand, Georgi~E. Shilov. 
\newblock {\em Generalized Functions. Vol I: Properties and Operations}. 
\newblock Academic Press Inc., New York, 1968.

\bibitem{Graf04}
Loukas Grafakos.
\newblock {\em Classical and modern fourier analysis}.
\newblock AMC, 10:12, 2004.

\bibitem{Hilf03}
Rudolf Hilfer.
\newblock On fractional diffusion and continuous time random
walks.
\newblock {\em Phys. A},  329:35--40, 2003.

\bibitem{IgnaRoss09}
Liviu~I. Ignat and Julio~D. Rossi.
\newblock Decay estimates for nonlocal problems via energy methods.
\newblock {\em Journal de math{\'e}matiques pures et appliqu{\'e}es},
  92(2):163--187, 2009.
  
\bibitem{Kass09} Moritz Kassmann.
\newblock A priori estimates for integro-differential operators with measurable kernels. 
\newblock {\em Calc. Var. Partial Differential
Equations} (34): 1--21, 2009.

\bibitem{KempSiljVergZach14}
Jukka Kemppainen, Juhana Siljander, Vicente Vergara, and Rico Zacher.
\newblock Decay estimates for time-fractional and other non-local in time
  subdiffusion equations in {$\mathbb R^d$}.
\newblock Submitted, 2014.

\bibitem{KilbSaig04}
Anatoly A.~Kilbas and Megumi Saigo.
\newblock {\em H-transforms. Theory and Applications.}
\newblock Analytical Methods and 
Special Functions, 9. Charman \& Hall/CRC, 390:1, 2004.

\bibitem{KimLee13}
Yong-Cheol Kim and Ki-Ahm Lee.
\newblock Regularity results for fully nonlinear parabolic integro-differential
  operators.
\newblock {\em Math. Ann.}, 357(4):1541--1576, 2013.

\bibitem{KimLim15}
Kyeong-Hun Kim and Sungbin Lim.
\newblock Asymptotic behaviors of fundamental solution and its derivatives related to space-time fractional differential equations.
\newblock {\em arXiv preprint arXiv:1504.07386v4}, 2015.

\bibitem{Koch90}
Anatoly~N. Kochubei.
\newblock Fractional-order diffusion.
\newblock {\em Differ. Equ.}, 26(4):485--492, 1990.

\bibitem{MeerBensScheBeck02}
Mark~M. Meerschaert, David~A. Benson, Hans-Peter Scheffler, and Peter
  Becker-Kern.
\newblock Governing equations and solutions of anomalous random walk limits.
\newblock {\em Phys. Rev. E}, 66:060102, Dec 2002.

\bibitem{MetzKlaf00}
Ralf Metzler and Joseph Klafter.
\newblock The random walk's guide to anomalous diffusion: a fractional dynamics
  approach.
\newblock {\em Physics Reports}, 339(1):1 -- 77, 2000.

\bibitem{MillSamk01}
Kenneth~S. Miller and Stefan~G. Samko.
\newblock Completely monotonic functions.
\newblock {\em Integral Transforms and Special Functions}, 12(4):389--402,
  2001.

\bibitem{Prus93}
Jan Pr{\"u}ss.
\newblock {\em Evolutionary integral equations and applications}, volume~87 of
  {\em Monographs in Mathematics}.
\newblock Birkh\"auser Verlag, Basel, 1993.

\bibitem{Ross09}
Julio~D. Rossi.
\newblock Asymptotics for evolution problems with nonlocal diffusion.
\newblock Manuscript available at
  http://mate.dm.uba.ar/~jrossi/CURSO(Marra)25-3-08.pdf, 2009.

\bibitem{Scho38}
Isaac~J. Schoenberg.
\newblock Metric spaces and completely monotone functions.
\newblock {\em Annals of Mathematics}, pages 811--841, 1938.

\bibitem{Silv07}
Luis Silvestre.
\newblock Regularity of the obstacle problem for a fractional power of the
  laplace operator.
\newblock {\em Communications on pure and applied mathematics}, 60(1):67--112,
  2007.

\bibitem{Silv05}
Luis~Enrique Silvestre.
\newblock {\em Regularity of the obstacle problem for a fractional power of the
  {L}aplace operator}.
\newblock ProQuest LLC, Ann Arbor, MI, 2005.
\newblock Thesis (Ph.D.)--The University of Texas at Austin.

\bibitem{TerERobiSikoZhu07}
A.F.M. Ter Elst, Derek W. Robinson, Adam Sikora and Yueping Zhu.
\newblock Second-order operators with degenerate coefficients.
\newblock {\em Proceedings of the London Mathematical Society}, 95(2):299--328,
  2007.


\bibitem{Vazq14}
Juan~Luis V{\'a}zquez.
\newblock Barenblatt solutions and asymptotic behaviour for a nonlinear
  fractional heat equation of porous medium type.
\newblock {\em Journal of the European Mathematical Society}, 16(4):769--803,
  2014.

\bibitem{VergZach15}
Vicente Vergara and Rico Zacher.
\newblock Optimal decay estimates for time-fractional and other nonlocal
  subdiffusion equations via energy methods.
\newblock {\em SIAM Journal on Mathematical Analysis}, 47(1):210--239, 2015.

\bibitem{Wats95}
George~Neville Watson.
\newblock {\em A treatise on the theory of Bessel functions}.
\newblock Cambridge university press, 1995.

\bibitem{Zach05}
Rico Zacher.
\newblock Maximal regularity of type $L_p$ for
abstract parabolic Volterra equations.
\newblock {\em J. Evol. Equ.}, 5:79--103, 2005.

\bibitem{Zach08}
Rico Zacher. 
\newblock Boundedness of weak solutions to evolutionary partial
integro-differential equations with discontinuous coefficients.
\newblock {\em J. Math. Anal. Appl.}, 348:137--149, 2008.

\bibitem{Zach09}
Rico Zacher.
\newblock Weak solutions of abstract evolutionary integro-differential
  equations in {H}ilbert spaces.
\newblock {\em Funkcial. Ekvac.}, 52(1):1--18, 2009.

\bibitem{Zach13b}
Rico Zacher.
\newblock A {D}e {G}iorgi--{N}ash type theorem for time fractional diffusion
  equations.
\newblock {\em Math. Ann.}, 356(1):99--146, 2013.

\bibitem{Zach13a}
Rico Zacher.
\newblock A weak {H}arnack inequality for fractional evolution equations with
  discontinuous coefficients.
\newblock {\em Ann. Sc. Norm. Super. Pisa Cl. Sci. (5)}, 12(4):903--940, 2013.

\bibitem{Zuaz03}
Enrique Zuazua.
\newblock Large time asymptotics for heat and dissipative wave equations.
\newblock Manuscript available at http://www.uam.es/enrique.zuazua, 2003.

\end{thebibliography}

\end{document}